\documentclass[12pt,a4paper,reqno]{article}
\usepackage{amsmath}
\usepackage{amsthm}
\usepackage{enumerate}
\usepackage{amssymb}

\usepackage{verbatim}
\usepackage{url}
\usepackage[latin1]{inputenc}
\usepackage{ulem}

\usepackage{caption}
\usepackage{graphicx,color}

\def\Pr{\textup{P}}

\def\be{\begin{equation}}
\def\ee{\end{equation}}
\def\bea{\begin{equation*}}
\def\eea{\end{equation*}}
\def\begs{\begin{split}}
\def\ends{\end{split}}

\newtheorem{thm}{Theorem}
\newtheorem{lma}[thm]{Lemma}

\newtheorem{prop}[thm]{Proposition}
\newtheorem{claim}{Claim}

\newtheorem*{thm*}{Theorem}
\newtheorem*{prop*}{Proposition}

\theoremstyle{remark}
\newtheorem{preremark}{Remark}
\newtheorem{preex}[thm]{Example}

\newenvironment{remark}{\begin{preremark}}{\end{preremark}}

\numberwithin{equation}{section}

\title{Arm events in two-dimensional invasion percolation}
\author{Michael Damron \thanks{The research of M. D. is supported by NSF grant DMS-0901534 and an NSF CAREER grant.} \\ \small{Georgia Tech}  \and Jack Hanson \thanks{The research of J. H. is supported by NSF grant DMS-1612921.}\\ \small{City College, CUNY} \and Philippe Sosoe\thanks{The research of P. S. is supported by the Center for Mathematical Sciences and Applications at Harvard University.} \\ \small{CMSA, Harvard}}

\begin{document}

\maketitle

\abstract{We compare the probabilities of ``arm events'' in two-dimensional invasion percolation to those in critical percolation. Arm events are defined by the existence of a prescribed ``color sequence'' of invaded and non-invaded connections from the origin to distance $n$. We find that, for sequences of a particular form, arm probabilities in invasion percolation and critical percolation are comparable, uniformly in $n$, while they differ by a power of $n$ for all others. A corollary of our results is the existence, on the triangular lattice, of arm exponents for invasion percolation, for any color sequence with at least two open (invaded) entries.}

\section{Introduction}

In this paper, we consider two-dimensional invasion percolation, a stochastic growth model, and compare probabilities of ``arm events'' to analogous ones in critical percolation. These events, defined by the occurrence of a sequence of disjoint invaded and non-invaded connections from the origin to a large distance, play a central role in percolation theory. 

The probabilities of arm events in critical and near-critical Bernoulli percolation have been studied extensively. In these settings, arm probabilities are expected to behave like power functions of the distance from the origin, and are thus naturally associated to the corresponding exponents. Existence and exact values of the exponents have been derived for the one-arm event \cite{lawlerschrammwerner}, as well as events associated to polychromatic sequences \cite{schrammwerner}, which include both open and closed dual connections, for percolation on the triangular lattice, using conformal invariance and the SLE processes. Existence of exponents for monochromatic arm sequences was proved in \cite{beffaranolin}, where it was also shown that the exponents differ from polychromatic exponents. To the best of our knowledge, arm probabilities in invasion percolation have not been studied previously.

Critical percolation is a natural point of comparison for invasion percolation because the latter process is an example of self-organized criticality, loosely defined as the tendency for a parameterless model to behave in the large time limit like a parametric model at its critical point. It has been shown, for instance, that the cluster volume, multi-point functions, and expected cluster sizes are of the same order for invasion and critical percolation \cite{DSV, jarai}.

Arm events are characterized up to permutation by a finite sequence specifying the status (invaded, resp. open, or non-invaded, resp. closed) and the order in which the connections appear. Our results imply that, for some color sequences, arm probabilities in invasion percolation are comparable to the corresponding arm probabilities in critical percolation. (See Theorem \ref{thm: at_least_two}.) This theorem also implies existence of exponents, in the sense of \cite{schrammwerner} on the triangular lattice, for all color sequences that include at least two invaded (open) arms. For all other color sequences, we show in Theorem \ref{thm: only_one_a} that the probabilities differ from those in critical percolation by a positive power of the distance. Combining the two theorems, we find in particular that polychromatic arm exponents of a given length in invasion percolation, if they exist, must generally depend on the exact color sequence. This is unlike critical (site) percolation, where a ``color-switching'' argument shows that arm probabilities are the same for all sequences of a given length; see \cite[Proposition 20]{nolin}.

\subsection{Invasion Percolation}
The invasion percolation cluster (IPC) is defined recursively as follows. Let $(t_e)$ be an i.i.d. family of uniform $[0,1]$ random variables, one assigned to each edge $e \in \mathcal{E}^2$ of the nearest-neighbor lattice $\mathbb{Z}^2$. Our growth model begins with vertex and edge set 
\[
V_0 = \{0\},~ E_0 = \emptyset,~ G_0 = (V_0,E_0).
\]
At each time $n$, let $e_n = \{x_n,y_n\}$ be the edge in the boundary
\[
\partial G_{n-1} = \{\{x,y\} : x \in V_{n-1}, \{x,y\} \notin E_{n-1}\}
\]
of $G_{n-1}$ such that $t_{e_n}$ is minimal, and define
\[
V_n = V_{n-1} \cup \{y_n\},~ E_n = E_{n-1} \cup \{e_n\},~G_n = (V_n,E_n).
\]
As $(G_n)$ is an increasing sequence of graphs, we define the IPC as $S = \lim_n G_n$.

To define critical percolation, we pick any $p \in [0,1]$ and say that the edge $e$ is $p$-open if $t_e < p$ and $p$-closed otherwise. A $p$-open cluster is a maximal set of vertices any two of which are connected by a path of $p$-open edges. For $p$ small, almost surely, there are only finite $p$-open clusters, whereas for $p$ large, there is a unique infinite $p$-open cluster. So one defines $\theta(p)$ as the probability that $0$ is in an infinite $p$-open cluster, and
\[
p_c = \sup\{p : \theta(p) = 0\}.
\]
We define the measure $\mathbb{P}_p$ on $\{0,1\}^{\mathcal{E}^2}$ as the one under which all coordinate functions $(\omega(e) : e  \in \mathcal{E}^2)$ are i.i.d. with
\[
\mathbb{P}_p(\omega(e) = 1) = p =  1 - \mathbb{P}_p(\omega(e)=0),
\]
and refer to an edge $e$ with $\omega(e) = 1$ simply as \emph{open}; otherwise, \emph{closed}. On $\mathbb{Z}^2$, it is known that $p_c = 1/2$, and we write $\mathbb{P}_{cr}$ for $\mathbb{P}_{1/2}$.

We now describe arm events, which are our main objects of study. A ``color sequence'' $\sigma$ is a sequence, each entry of which is either the symbol $O$ (for open) or $C$ (for closed); we write $|\sigma|_O$ for the number of open entries and $|\sigma|_C$ for the number of closed entries, with $|\sigma|$ defined to be the total number of entries. For $B(n) = [-n,n]^2$, we define a $\sigma$-connection between a vertex $x$ and the set 
\[
\partial B(n) = \{z \in B(n) : \exists y \in B(n)^c \text{ such that } \|z-y\|_1 = 1\}
\]  
as follows: write $x \to_\sigma \partial B(n)$ if there exist $|\sigma|$ disjoint paths from $B(|\sigma|) + x$ to $\partial B(n)$, $|\sigma|_O$ of which are open paths, and $|\sigma|_C$ of which are closed dual paths, and their orientation is given by $\sigma$. By dual path, we refer to the associated percolation model defined on the dual lattice $(\mathbb{Z}_*^2, \mathcal{E}_*^2)$, with vertex and edge sets
\[
\mathbb{Z}^2_* = \mathbb{Z}^2 + (1/2,1/2),~ \mathcal{E}_*^2 = \mathcal{E}^2 + (1/2,1/2).
\]
A dual edge $e^*$ bisects exactly one edge $e$, and we set $\omega(e^*) = \omega(e)$. A closed dual path is a sequence of closed edges $e$ whose dual edges form a path on the dual lattice. (These arm events are always considered up to cyclic permutation of $\sigma$.) A $\sigma$ connection is defined naturally for critical percolation in terms of open and closed edges; for the IPC, we interpret ``$e$ is open'' as meaning that $e \in S$, or that $e$ is eventually invaded. Likewise, ``$e$ is closed'' means that $e \notin S$, or that $e$ is never invaded. For $1 \leq m \leq n$, the notation $\partial B(m) \to_\sigma \partial B(n)$ is used similarly: there is a $\sigma$-connection between $\partial B(m)$ and $\partial B(n)$. (If $m \leq |\sigma|$, then the inner box is taken to be $B(|\sigma|)$.)

For {two positive} sequences $(a_n)$ and $(b_n)$, write $a_n \lesssim b_n$ if the ratio $a_n/b_n$ is bounded as $n \to \infty$. If $a_n \lesssim b_n$ and $b_n \lesssim a_n$, we write $a_n \asymp b_n$.

\subsection{Main results}

If $\sigma$ is any color sequence, we define the reduced color sequence $\tilde \sigma$ by replacing any consecutive stretch of at least two `$C$' entries by two `$C$' entries. Our first theorem states that $\sigma$-arm events in invasion have probability comparable to those in critical percolation for the reduced color sequence $\tilde \sigma$, provided that $|\sigma|_O \geq 2$. Note that if $\sigma$ has no stretches of more than three `$C$' entries (for instance, in the open monochromatic case), then $\sigma = \tilde \sigma$, and so $\sigma$-arm events have comparable probability in both models.

Let $A_\sigma(n)$ be the event
\begin{equation}\label{eq: a_sigma_def}
A_\sigma(n) = \{0 \to_\sigma \partial B(n)\}.
\end{equation}
When we write $\mathbb{P}(A_\sigma(n))$, it is understood that the event $A_\sigma(n)$ is taken in the IPC $S$, so that, for example, ``open'' means ``invaded.'' As the IPC has an open path from 0 to $\infty$ almost surely, we consider only sequences with $|\sigma|_O\geq 1$. We now present the first theorem. Note that combined, both of its statements imply that if $|\sigma|_O \geq 2$, then
\[
\mathbb{P}_{cr}(A_{\tilde \sigma}(n)) \asymp \mathbb{P}(A_\sigma(n)).
\]
\begin{thm}\label{thm: at_least_two}
Let $\sigma$ be a color sequence with $|\sigma|_O\geq 1$. Then
\begin{equation}\label{eqn: thm1-lwr-bd}
\mathbb{P}_{cr}(A_{\tilde \sigma}(n)) \lesssim  \mathbb{P}(A_\sigma(n)).
\end{equation}
If $|\sigma|_O\geq 2$ then 
\begin{equation}\label{eqn: thm1-upper-bd}
\mathbb{P}(A_\sigma(n)) \lesssim \mathbb{P}_{cr}(A_{\tilde \sigma}(n)).
\end{equation}
\end{thm}
\begin{remark}
It follows immediately from Theorem \ref{thm: at_least_two} that if $\sigma$ has a least two $O$ entries and no more than two consecutive $C$ entries, the probability of the corresponding arm events in invasion and critical percolation are comparable, uniformly in $n$.
\end{remark}
\begin{remark}
For any sequence $\sigma$ with $|\sigma|_O\ge 2$, Theorem \ref{thm: at_least_two} provides matching upper and lower bounds for the probability of the arm events.
Our proofs are easily adapted to the triangular lattice, and in this setting we obtain the existence of arm exponents in the sense of \cite[Theorem 4]{schrammwerner}.
\end{remark}

All color sequences covered in the next theorem are of the form 
\[
\sigma_k := (O, \underbrace{C, C, \ldots, C}_{k \, \text{times}}).
\]
For such arm events, we have two main bounds. The first inequality is strictly stronger than the first inequality of the previous theorem, and shows that the arm probabilities (even in the cases $k \leq 2$, where $\sigma_k = \tilde{\sigma}_k$) are not comparable in critical percolation and in the IPC.
\begin{thm}\label{thm: only_one_a}
There exists $\epsilon>0$ such that for all  $k \geq 1$,
\begin{equation}\label{eqn: thm2-first}
\mathbb{P}_{cr}(A_{\tilde{\sigma}_k}(n)) n^\epsilon \lesssim \mathbb{P}(A_{\sigma_k}(n)).
\end{equation}
For $k=1,2$, there exists $\epsilon>0$ such that
\begin{equation} \label{eqn: thm2-two}
\mathbb{P}(A_{\sigma_k}(n)) n^{\epsilon} \lesssim \mathbb{P}_{cr}(A_{\hat{\sigma}_k}(n)),
\end{equation}
where $\hat \sigma_k$ is a sequence of $k$ `$C$' entries and zero `$O$' entries.
\end{thm}
The proof of the lower bound, inequality \eqref{eqn: thm2-first}, uses a new strategy to show that conditional on an arm event to distance $n$, there is high probability for occurrence of at least order $\log n$ many ``outlet'' events defined in disjoint annuli. This method, detailed in Section~\ref{sec: thm2-lwr-bd}, is a crucial tool for our upcoming work on a strict inequality for the chemical distance exponent in critical percolation \cite{new_paper}. We believe that it could have further applications.

\begin{remark}
If an exponent $\alpha_k$ exists for the invasion arm event $A_{\sigma_k}(n)$, then by applying the previous theorem, one has
\[
\alpha_2 > \beta_2',~ \beta_2 > \alpha_1,~\text{and } \beta_3 - \epsilon \geq \alpha_k \text{ for } k \geq 2,
\]
where $\beta_2'$ is the monochromatic $(CC)$ exponent, $\beta_2$ is the polychromatic $(OC)$ exponent, and $\beta_3$ is the polychromatic $(OCC)$ exponent for critical percolation. It is believed (and proved on the hexagonal lattice \cite{beffaranolin}) that $\beta_2' > \beta_2$, so this would give $\alpha_2 > \alpha_1$, implying
\[
0 < \alpha_1 < \alpha_2 \leq \alpha_3 \leq \cdots \leq \beta_3 - \epsilon < \beta_3.
\]
It is reasonable, based on our first theorem, to believe that $\alpha_2 = \alpha_3 = \alpha_4 = \cdots$, but we do not yet have a proof.
\end{remark}

\begin{remark}
The Incipient Infinite Cluster (\emph{IIC}), is the measure defined by the limit of the conditional probabilities
\[\lim_{N\rightarrow \infty} \mathbb{P}(A\mid 0 \rightarrow \partial B(N)).\]
The limit was shown to exist in \cite{kesteniic} for cylinder sets $A$ and determines a measure $\mathbb{P}_{\mathrm{IIC}}$ on configurations in $\{0,1\}^{\mathcal{E}^2}$. An simple decoupling argument shows that for $|\sigma|_O \geq 1$,
\[
\mathbb{P}_{\mathrm{IIC}}(A_\sigma(n))\asymp \frac{\mathbb{P}_{cr}(A_\sigma(n))}{\mathbb{P}_{cr}(0 \rightarrow \partial B(n))}.
\]
It has been observed that the IIC measure, although distinct from the invasion measure, is similar in many respects \cite{DSV, jarai}. It may be possible to sharpen the results in Theorem \ref{thm: only_one_a} by comparing to IIC probabilities, instead of probabilities in critical percolation.
\end{remark}

\subsection{Notation}

Here we review notation used throughout the paper. Recall that for $n \geq 1$, $B(n) = [-n,n]^2$. For $1 \leq m \leq n$, set
\[
Ann(m,n) = B(n) \setminus B(m).
\]

Recall the definition of a $\sigma$-connection between $\partial B(m)$ and $\partial B(n)$ from the introduction, and that we are working on two different probability spaces. The first contains i.i.d. uniform random variables $(t_e)$ assigned to the edges and a corresponding measure $\mathbb{P}$. It is on this space that the IPC $S$ is defined, and on which we have notions of $p$-openness. The second is the space $\{0,1\}^{\mathcal{E}^2}$, with probability measures $\mathbb{P}_p$ under which the coordinates are i.i.d. with probability $p$ to be 1. On this first space, we define both events
\begin{align*}
A_\sigma(n) &= \{0 \to_\sigma \partial B(n) \text{ in }S\} \\
A_\sigma(m,n) &= \{\partial B(m) \to_\sigma \partial B(n) \text{ in }S\}.
\end{align*}
We also set $A_\sigma(n,p,q)$ as the event that $0 \to_\sigma \partial B(n)$ but the open arms are $p$-open, and the closed arms are $q$-closed (similarly for $A_\sigma(m,n,p,q)$). Similarly, on the second space:
\begin{align*}
\text{on } \{0,1\}^{\mathcal{E}^2},~ A_\sigma(n) & = \{0 \to_\sigma \partial B(n)\} \\
A_\sigma(m,n)&= \{\partial B(m) \to_\sigma \partial B(n)\}.
\end{align*}

\subsection{Outline of the paper}

In deriving our main results, we exploit the natural coupling of independent (Bernoulli) percolation processes at different parameters and invasion percolation, and the ``near-critical'' nature of the invasion cluster on large scales. We summarize some of the relevant techniques and estimates from the literature in Section \ref{sec: tools}.

In Section \ref{sec: lwr-bd}, we derive the lower bound \eqref{eqn: thm1-lwr-bd} in Theorem \ref{thm: at_least_two}.  The proof proceeds by the construction of a sequence of arms according to the sequence $\tilde{\sigma}$ from the origin to a dual $p_n$-closed circuit around $0$ with one defect edge $e$ ($e$ is not required to be $p_n$-closed); $e$ is itself connected to one of the open arms, and to an infinite $p_n$-open cluster. (See \eqref{eq: p_n_def} for the definition of $p_n$ and Figure~\ref{arms_construction2} for an illustration.) This construction has a total probability cost of $\mathbb{P}_{cr}(A_{\tilde{\sigma}}(n))$.  The dual circuit and dual closed arms are chosen in such a way that any region between two closed arms and the circuit is not invaded. Additional non-invaded arms can be found in this region, provided there is sufficient space. To show that with high probability the closed dual arms cannot come too close to each other, we use the exact value of the five-arm exponent and Reimer's inequality (see Lemma \ref{lma: lma_1}) to bound the probability of a resulting six-arm event.

The proof of \eqref{eqn: thm1-upper-bd} appears in Section \ref{sec: upper-bd}. We decompose the probability of the arm event in invasion according to the position of the first \emph{outlet} $\hat{e}_1$, the invaded edge with maximal weight in the IPC. If $t_{\hat{e}_1} < \tau$, the entire invasion is made of $\tau$-open edges, and in particular, any invaded arms are $\tau$-open. As for non-invaded arms, Lemma \ref{lma: sigmareduction} shows that any region which contains one (resp. two) consecutive non-invaded arms from the origin must contain one (resp. two) $p_c$-closed arms. If $\hat{e}_1$ is near the origin, the boundary of the invasion at the step when $\hat{e}_1$ is added to the invasion graph contains long closed dual arms from $(\hat{e}_1)^*$ to $\partial B(n)$; see Lemma \ref{lma: sigmareduction}. This unlikely event generates a small probability factor allowing us to use an argument from \cite{jarai} to sum over a partition of the possible values of $\tau$
 without losing any logarithmic factors. If $\hat{e}_1$ is far from the origin, its weight is likely to be close to $p_c$. This is quantified using estimates for the correlation length from \cite{jarai}, and the same summation argument introduced in that paper.

The first inequality of Theorem \ref{thm: only_one_a}, \eqref{eqn: thm2-first}, is derived in Section \ref{sec: thm2-lwr-bd} by constructing outlets in annuli in a positive density of scales. Outlets are edges whose weight is greater than any edges invaded at a later stage. We connect these successive outlets by $p_c$-open paths, and this ensures that closed connections placed between dual circuits which have defects at the outlets remain non-invaded. By quasimultiplicativity, constructing the outlets and extending closed arms through all annuli has a probability cost of $\mathbb{P}_{cr}(A_{\sigma_k}(n))$ times a (very small) gluing factor depending on the number of outlets placed. The large number of choices of possible placements for the annuli then beats this gluing factor and yields the additional $n^\epsilon$.

The upper bound \eqref{eqn: thm2-two} is proved in Section \ref{sec: thm2-upper-bd}. Here, we show that the event $A_{\sigma_k}(n)$ implies $A_{\hat{\sigma}_k}(|\sigma_k|,n,p_c)$ when $k=1,2$, so, using \eqref{eq: changearms}, we have
\[
\mathbb{P}(A_{\sigma_k}(n)) \le C\mathbb{P}(A_{\sigma_k}(n)\mid A_{\hat{\sigma}_k}(|\sigma_k|,n,p_c))\mathbb{P}_{cr}(A_{\hat{\sigma}_k}(n)).
\]
The problem is then reduced to showing that the conditional probability is $O(n^{-\epsilon})$ for some $\epsilon>0$. By a characterization of the occurrence of an invaded circuit in an annulus in Lemma~\ref{lma: inv-circuit}, showing this bound reduces to finding with high probability, a $p$-open circuit around 0 enclosed by a $p$-closed dual circuit for some $p>p_c$. We estimate this conditional probability using a decoupling idea due to Kesten and Zhang \cite{kestenzhang}.

\section{Tools from near-critical percolation}
\label{sec: tools}
For any $p \in [0,1]$ and $n \geq 1$, define the box-crossing event $\sigma(n)$ that there is an open path in $B(n)$ connecting the left side $\{-n\} \times [-n,n]$ to the right side $\{n \} \times [-n,n]$. It is known that for $p > p_c$, $\mathbb{P}_p(\sigma(n)) \to 1$, and so, given $\epsilon>0$, we define the correlation length $L(p, \epsilon)$ of $p$ as
\[
L(p,\epsilon) = \min\{n \geq 1 : \mathbb{P}_p(\sigma(n)) > 1-\epsilon\}.
\]
Kesten \cite[Eq.~1.24]{kestenscaling} has shown that there is $\epsilon_0$ such that for any $\epsilon_1, \epsilon_2 \in (0,\epsilon_0]$, one has
\[
L(p,\epsilon_1) \asymp L(p,\epsilon_2) \text{ as } p \downarrow p_c,
\]
so we simply write $L(p)$ for $L(p,\epsilon_0)$.

It is useful to roughly invert the function $p \mapsto L(p)$, so for $n \geq 1$, we define
\begin{equation}\label{eq: p_n_def}
p_n = \min \{{p > p_c} : L(p) \leq n\}.
\end{equation}

Below, we list various properties of correlation length and ``classical results'' that we will use. Recalling the definition of $A_\sigma(n)$ in \eqref{eq: a_sigma_def}, we extend this definition to $A_\sigma(n,p,q)$ for $n \geq 1$ and $p,q \in [0,1]$ as the event that $\partial B(|\sigma|)$ is connected to $\partial B(n)$ by $|\sigma|_O$ $p$-open paths, $|\sigma|_C$ $q$-closed duals paths, and these paths occur in the orientation given by $\sigma$.
\begin{itemize}
\item For any $\sigma$, there is a constant $C>0$ such that, for $m \leq n\le \min\{L(p_1),L(p_2)\}$, we have the estimate:
\begin{equation}\label{eq: changearms}
(1/C) \mathbb{P}_{cr}(A_{\sigma}(m,n)) \leq \mathbb{P}(A_{\sigma}(m,n, p_1, p_2))\le  C\mathbb{P}_{cr}(A_{\sigma}(m,n)).
\end{equation}
This is proved in \cite[Theorem 27]{nolin} when $p_1=p_2$. The proof given there also applies to our case. See \cite[Lemma 6.2]{DSV} for a proof in case $\sigma=(O,C,O,C)$.
\item There is a constant $C >0$ such that for all $n$,
\begin{equation}\label{eq: p_n_to_infty}
\mathbb{P}(\partial B(n) \to \infty \text{ by a }p_n\text{-open path}) \geq C.
\end{equation}
See \cite[Theorem 2]{kestenscaling}.




\item (Similar to \cite[Eq.~(2.25)]{kestenscaling}.) There are constants $C_1,C_2>0$ such that for all $p > p_c$ and all $n \geq L(p)$,
\begin{equation}\label{eq: closed_exponential}
\mathbb{P}(\exists \text{ a } p\text{-closed dual path of diameter at least }n/10 \text{ in } B(n)) \leq C_1e^{-C_2\frac{n}{L(p)}}.
\end{equation}

\item By \cite[Eq.~(2.10)]{jarai}, there is $D\in(0,1)$ such that for all $n$,
\begin{equation}\label{eq: D_p_n}
Dn \leq L(p_n) \leq n.
\end{equation}

\item We will make use of a slight generalization of the  Harris-FKG association inequality, sometimes referred to as the ``generalized FKG inequality''. This will allow the usual FKG techniques to be extended to some cases where events are not monotone; we will describe an example below when outlining ``gluing'' arguments. 

The version we give is from \cite[Lemma 13]{nolin}. Suppose $A_1$ and $A_2$ are increasing events in the space of the $(t_e)$ variables. That is, suppose that for each configuration $(t_e) \in A_1$ and each $(t_e')$ satisfying $t_e' \geq t_e$ for all $e$, we have $(t_e') \in A_1$ (and similarly for $A_2$). Suppose that $B_1$ and $B_2$ are two decreasing events. Last, suppose that there are finite disjoint subsets $\Xi_+$, $\Xi_-,$ $\Xi \subseteq \mathcal{E}^2$ such that $A_1$ depends only on the edge variables in $\Xi_+ \cup \Xi$, $B_1$ on the variables in $\Xi_- \cup \Xi$, $A_2$ on the variables in $\Xi_+$, and $B_2$ on the variables in $\Xi_-$. Then
\begin{equation}
  \label{eq:genfkg}
  \mathbb{P}\left(A_2 \cap B_2 \mid A_1 \cap B_1 \right) \geq \mathbb{P}(A_2) \mathbb{P}(B_2)\ .
\end{equation}

\item The Russo-Seymour-Welsh (RSW) theorem \cite{russo, seymourwelsh} says that, at criticality, open and closed box crossings are likely on all scales. We must extend our definition of $\sigma(n)$, defining $\sigma(m,n)$ to be the event that there is an open path in the rectangle $[-m,m]\times [-n,n]$ connecting the left side of this rectangle to the right side. With this notation, RSW says that for any $c > 0$, there is a $\delta_c > 0$ such that 
\begin{equation}
\label{eq:rswstate}
\delta_c \leq \liminf_n \mathbb{P}_{cr}(\sigma(n,cn)) \leq \limsup_n \mathbb{P}_{cr}(\sigma(n, cn)) \leq 1 - \delta_c\ . \end{equation}
In fact, $\delta_c$ can be chosen so that the bounds of \eqref{eq:rswstate} hold uniformly in the percolation parameter $p$, up to the correlation length:
\begin{equation}
  \label{eq:rswstate2}
  \delta_c \leq \inf_{n: \, \max\{n,\, cn\} \leq L(p)} \mathbb{P}_{p}(\sigma(n,cn)) \leq \sup_{n: \, \max\{n,\, cn\} \leq L(p)} \mathbb{P}_{p}(\sigma(n, cn)) \leq 1 - \delta_c\ .
\end{equation}

A useful graph-theoretic fact is that there is an open left-right crossing of the box $[0,m] \times [0,n]$ if and only if there is no closed dual top-bottom crossing of the dual box $[1/2,m-1/2] \times [-1/2,n+1/2]$. This observation, combined with \eqref{eq:rswstate} and \eqref{eq:rswstate2} (along with the symmetries of the lattice) give analogous RSW bounds for closed dual crossings.

\item We will often require use of standard ``gluing'' and ``arm separation'' techniques to extend and connect arms in critical and near-critical percolation. These techniques originated in \cite{kestenscaling} and have the RSW estimates as their chief ingredient. Because we will often omit details in the particular cases where these tools are used, we give here a typical application of these techniques (for more detail, see \cite[Section 4]{nolin}).

Fix the particular color sequences $\sigma_1 = (C), \, \sigma_2 = (O,C), \, \sigma_3 = (O)$. We claim that 
\begin{equation}
\label{eq:glueasymp}
\mathbb{P}_{cr}(A_{\sigma_1}(5n), \, 0 \leftrightarrow 3n e_1) \asymp \mathbb{P}_{cr}(A_{\sigma_2}(n)) \mathbb{P}_{cr}(A_{\sigma_3}(n))\ .  
\end{equation}
It is clear by inclusion and independence that the left-hand side of \eqref{eq:glueasymp} is bounded above by the right-hand side, so we focus on the lower bound. Note that the right-hand side is the probability of the event that $0$ has an open and closed connection to distance $n$ and that $3n e_1$ has an open connection to distance $n$. The intuition behind the bound is the following: given these ``distance $n$'' connections, there is at least some uniform (in $n$) constant probability that the closed connection from $0$ is connected to a closed rectangle crossing connecting $B(n)$ to $\partial B(5n)$ (the existence of such a crossing is furnished by the RSW technology). Similarly, the open connections from $0$ and $3n e_1$ have bounded conditional probability of connecting to the same open rectangle crossing connecting $B(n)$ to $B(n) + 3n e_1$.

To make this precise, we mandate the connections from $0$ to $\partial B(n)$ and $3n e_1$ to $3 n e_1 + \partial B(n)$ have particular ``landing sites''. That is, we ask that we can choose the closed arm from $0$ to connect to $\partial B(n)$ specifically in the segment $\{-n\} \times [-\delta n, \delta n]$ for some $\delta$ small, and the open arm to connect to $\{n\} \times [-\delta n, \delta n]$. 

We also ask that the arms not come too close to each other's endpoints: namely, that the closed arm not enter $[n - \delta n, n] \times [- \delta n, \delta n]$, and that the open arm not enter $[-n,-n+\delta n]\times[-\delta n, \delta n]$. We also mandate the arm from $3 n e_1$ have its endpoint in $\{2n\} \times [-\delta n, \delta n]$. This can be done using small scale ``fences'' -- rectangle crossings near arm endpoints -- to insulate the arms from each other and then direct them. 

Last, we ask for each of the separated arms to be ``extensible''. Taking the open arm from $0$ as an example and letting $z$ denote its endpoint, this means that for some fixed small $\varepsilon > 0$ there is an open vertical crossing of $z + [0, \varepsilon n] \times[-\varepsilon n, \varepsilon n]$ which is connected to the open arm by part of an open circuit; similar constructions are used on the remaining arms.  Figure \ref{fig:extend} depicts these arm extensions near the boundary of the small squares. The resulting event has probability at least $c \mathbb{P}_{cr}(A_{\sigma_2}(n)) \mathbb{P}_{cr}(A_{\sigma_3}(n))$ for some small uniform $c$; see \cite[Theorem 11]{nolin}.

Condition on the above event. We can now attach crossings of rectangles of diameter order $n$ to the small extensions of the preceding paragraph to direct the arms appropriately. Given the RSW estimates and generalization of FKG given at \eqref{eq:genfkg} (here we require the fixed, disjoint landing sites), these rectangle crossings can be shown have at least constant conditional probability. For instance, letting $z'$ denote the endpoint of the open arm from $3n e_1$, we ask for there to be an open left-right crossing of the rectangle $z' + [n, 2n] \times [-\varepsilon n, \varepsilon n]$. The existence of this rectangle crossing guarantees occurrence of $\{0 \leftrightarrow 3n e_1\}$; a similar extension is used on the closed arm. These rectangle crossings are depicted connecting the two squares (for the open crossing) and connecting the left square outward toward $\partial B(5n)$ (for the closed crossing) in Figure \ref{fig:extend}.

When applying the generalized FKG inequality above, one should take as $A_1$ and $B_1$ the events guaranteeing the existence of well-separated open and closed arms (respectively) from $0$ and $3 n e_1$. The events $A_2$ and $B_2$ can be taken to be the existence of the open and closed ``extensions'' of these arms; because these extensions are chosen localized in disjoint regions, we obtain the appropriate disjointness for the edges on which these events depend. The resulting event has probability at least $c \mathbb{P}_{cr}(A_{\sigma_2}(n)) \mathbb{P}_{cr}(A_{\sigma_3}(n))$ for some possibly smaller $c$ and implies $A_{\sigma_1}(5n) \cap \{0 \leftrightarrow 3 n e_1\}$, completing the proof.


\begin{figure}
\centering
\includegraphics[scale = 0.5]{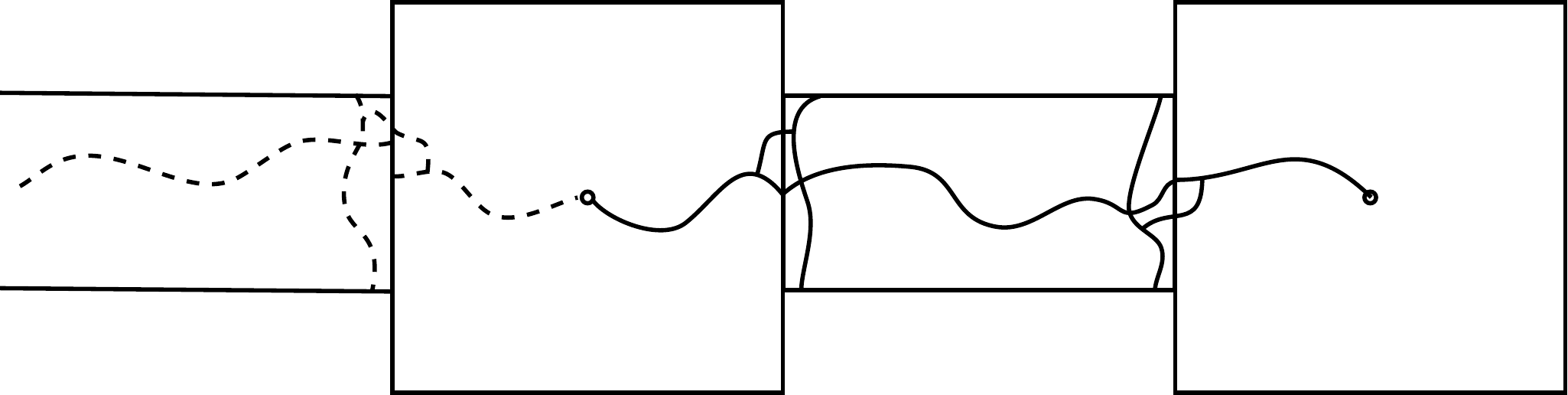}
\caption{Depiction of the gluing argument outlined above. The closed arm from $0$ is extended all the way to $\partial B(5n)$ (not shown). Note that the arms in $B(n)$ may wander throughout the box, though they may not come close to each other's endpoints.}
\label{fig:extend}
\end{figure}

\item One consequence of the above gluing techniques is the ``quasi-multiplicativity'' of arm events. This says that, up to constants, an arm event can be decomposed into two smaller-scale arm events. Fix some color sequence $\sigma$. Then uniformly in $p \geq p_c$ and $|\sigma| \leq k <  m < n \leq L(p)$, we have
\begin{equation}
\label{eq:quasmult}
\mathbb{P}_p(A_\sigma(k,n)) \asymp_\sigma \mathbb{P}_p(A_\sigma(k,m)) \mathbb{P}_p(A_\sigma(m,n))\ .    \end{equation}
Here ``$\asymp_\sigma$'' means that the ratio of the two sides is uniformly bounded by constants which depend possibly on $\sigma$ but not on $k\,,m,\,n,\,p$.

\end{itemize}

\section{Proof of Theorem~\ref{thm: at_least_two}}


Given a color sequence $\sigma$, recall the definition of the reduced sequence $\tilde{\sigma}$. It is the sequence obtained from $\sigma$ by replacing any subsequence of $l$ consecutive `$C$' entries by $\min(l,2)$ `$C$' entries. For example, if
\[\sigma = (O,C,C,C,O,C),\]
then
\[\tilde{\sigma} = (O,C,C,O,C),\]
whereas $\tau =(O,C,O,C)$ is equal to $\tilde{\tau}$.

\subsection{Proof of the lower bound}\label{sec: lwr-bd}

Here we prove the inequality
\begin{equation}\label{eq: one_bound}
\mathbb{P}_{cr}(A_{\tilde{\sigma}}(n))\lesssim \mathbb{P}(A_\sigma(n))
\end{equation}

\begin{proof}

We will first show that with positive probability, on the event $A_{\tilde \sigma}(n,p_c,p_n)$, we can force the existence of an outlet for the invasion in the annulus $Ann(n,2n)$. An outlet is an edge $e \in S$ such that if $e = e_s$ (the $s$-th invaded edge), then $t_e > \sup\{t_{e_r} : r > s\}$. The event $O_e$ in the next lemma does not directly imply that the edge $e$ is an outlet, but if it is combined with the conditions (a) one endpoint of $e$ is invaded and (b) the other endpoint of $e$ is connected to $\infty$ by a $p_n$-open path, then $e$ is an outlet. We mainly give the lemma and its proof to illustrate the general technique used to build four-arm edges that can be turned into outlets.

In the conditions below, items 1 and 2 imply that $e^*$ is part of a $p_n$-closed dual circuit $\mathcal{C}_e$ around 0 with one defect (an edge which is not $p_n$-closed), the edge $e$. See Figure \ref{arms1} for an illustration.

\begin{figure}
\centering
\includegraphics[scale = 0.35]{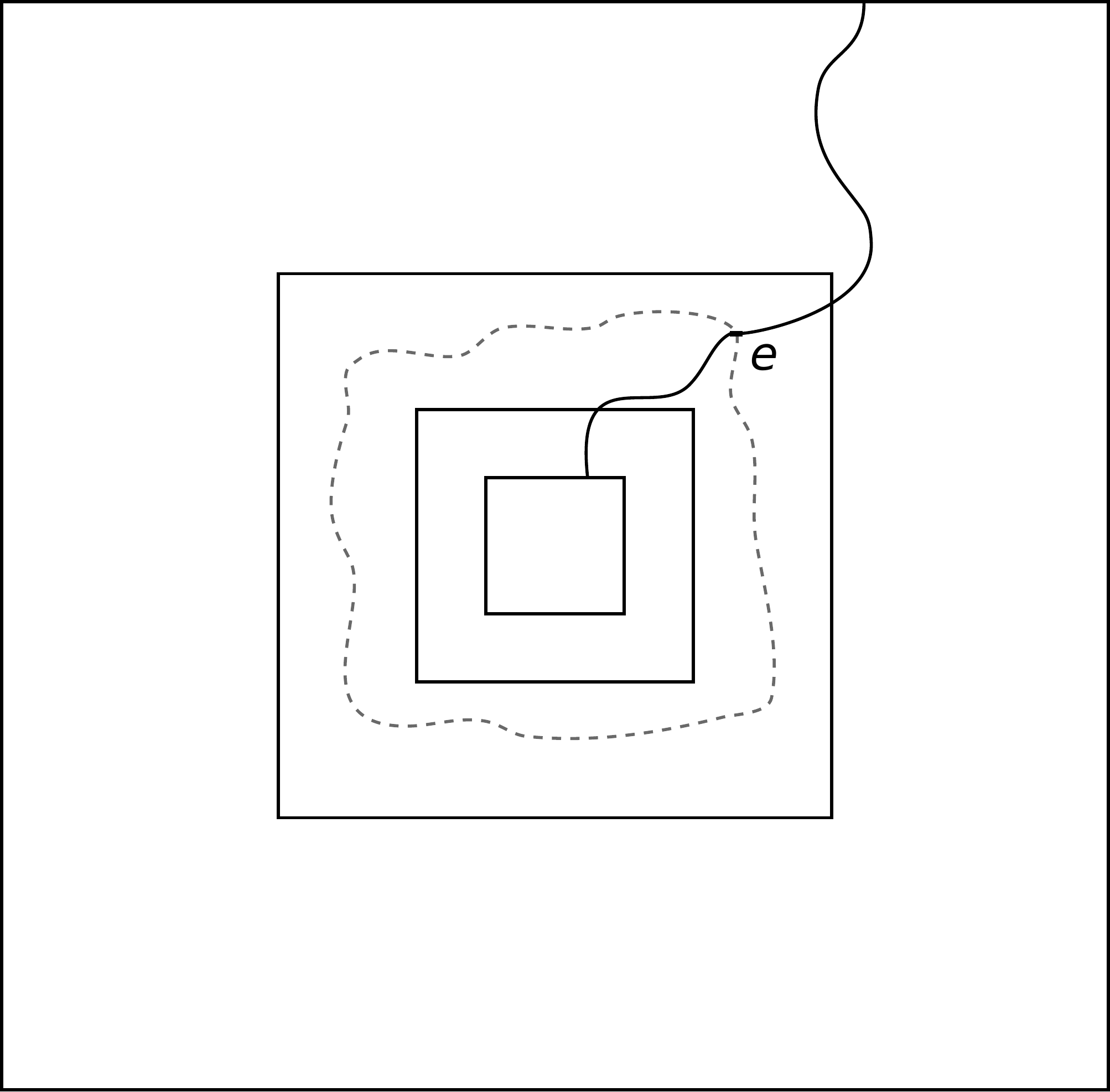}
\caption{The outlet construction in Lemma \ref{lma: outlet_construction}. The edge $e$ crosses the defect in an otherwise $p_n$-closed dual circuit $\mathcal{C}_e$. The open paths from $e$ are $p_c$-open; if these connections have $p_n$-open extensions to $0$ and $\infty$, then $e$ will be an outlet. \label{fig:outcons}}
\label{arms1}
\end{figure}

\begin{lma}\label{lma: outlet_construction}
For $e \in Ann(n,2n)$, let $O_e$ be the event that the following conditions hold:
\begin{enumerate}
\item $t_e \in (p_c,p_n)$,
\item one endpoint of the dual edge $e^*$ is connected by a $p_n$-closed dual path in $Ann(n,2n)$ around the origin to the other endpoint of $e^*$,
\item one endpoint of $e$ is connected to $\partial B(4n)$ by a $p_c$-open path, and
\item the other endpoint of $e$ is connected to $\partial B(n/2)$ by a $p_c$-open path.
\end{enumerate}
There exists $C>0$ such that for all $n$,
\[
\mathbb{P}\left( \bigcup_{e \in Ann(n,2n)} O_e \right)  \geq C.
\]
\end{lma}
\begin{proof}
By \eqref{eq: changearms},
\[
\mathbb{P}_{cr}(A_\tau(4n)) \leq C \mathbb{P}(A_\tau(4n,p_c,p_n)),
\]
where $\tau = (OCOC)$. By directing these arms as in \cite{kestenscaling} with the generalized FKG inequality and the RSW theorem, one obtains $C$ such that for all $n$ and $e \in Ann(n,2n)$,
\[
\mathbb{P}(O_e) \geq C \mathbb{P}_{cr}(A_\tau(4n)) (p_n-p_c).
\]
Since the events $(O_e)$ are disjoint,
\begin{align*}
\mathbb{P}(\cup_e O_e) = \sum_e \mathbb{P}(O_e) &\geq C n^2 \mathbb{P}_{cr}(A_\tau(4n)) (p_n-p_c) \\
&\geq C n^2 \mathbb{P}_{cr}(A_\tau(n)) (p_n-p_c).
\end{align*}
In the last inequality we have used quasimultiplicativity. Last, we use the scaling relation \cite[Prop.~34]{nolin}
\[
n^2 \mathbb{P}_{cr}(A_\tau(n)) (p_n-p_c) \asymp 1.
\]
\end{proof}

We will naturally want to force an edge $e$ for which $O_e$ occurs to be an outlet. We also want to do this in such a way that there are $\tilde \sigma$ arms from $\partial B(|\tilde \sigma|)$ to $\partial B(n)$. To do so, it will be easier to begin the arms from some fixed box $B(K)$ instead of near 0. So for $n \geq 8K \geq 8|\tilde {\sigma}|$, we consider the event $A_{\tilde \sigma}^*(K,n,p_c,p_n)$ comprised of the following conditions (see Figure \ref{arms_construction2}):

\begin{figure}
\centering
\includegraphics[scale = 0.35]{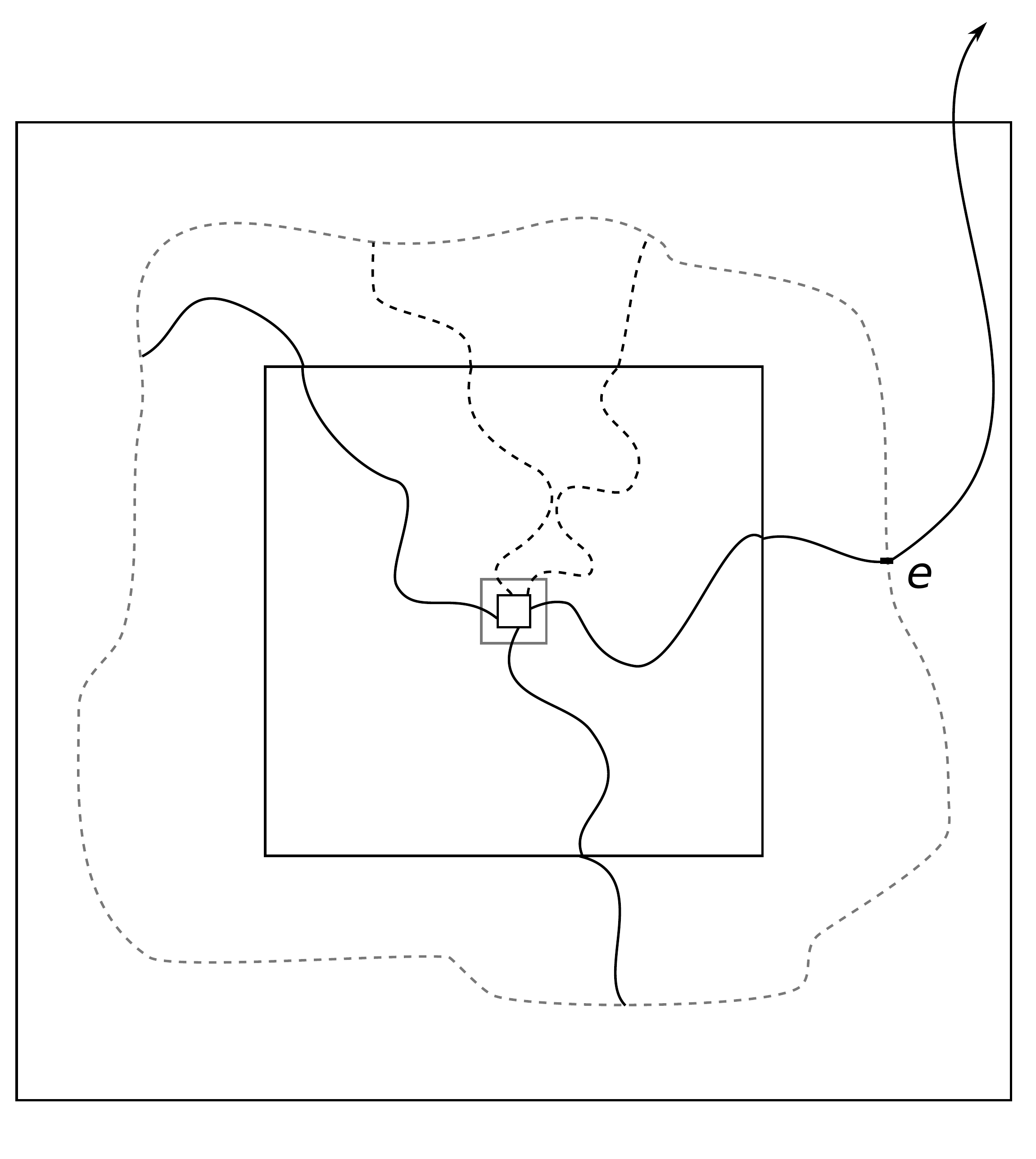}
\caption{The event $A_{\tilde \sigma}^*(K, n, p_c, p_n)$. The construction is as in Figure \ref{fig:outcons}, with additional connections from $B(K)$ (the innermost box) to the closed defected circuit crossed by $e$, and a $p_n$-open extension of the path from $e$ which connects $e$ to $\infty$ (in the annulus $Ann(n,2n)$). Note that $e$ is not yet guaranteed to be an outlet; this will be enforced by imposing conditions on the weights in $B(K)$.}
\label{arms_construction2}
\end{figure}

\begin{enumerate}
\item there is an edge $e$ in $Ann(n,2n)$ for which $O_e$ occurs,
\item there is a $p_n$-open path connecting $e$ to $\infty$,
\item there is a $\tilde{\sigma}$ connection from $B(K)$ to $\partial B(n)$; the open paths are $p_c$-open, the closed dual paths are $p_n$-closed, one of the open paths further connects to $e$, the other open paths further connect to edges dual to those on the $p_n$-closed circuit $\mathcal{C}_e$ with one defect from the definition of $O_e$, and the closed dual paths further connect to the same circuit.
\end{enumerate}

Note that the $p_n$-open path from item 2 is not necessarily disjoint from the $p_c$-open path from item 3 of the definition of $O_e$, and the $p_c$-open path from the $\tilde \sigma$-connection in item 3 connecting to $e$ is not necessarily disjoint from the $p_c$-open path from item 4 of the definition of $O_e$.

We next claim that
\begin{equation}\label{eq: tostada}
\mathbb{P}(A_{\tilde \sigma}(K,n,p_c,p_n)) \leq C\mathbb{P}(A_{\tilde \sigma}^*(K,n,p_c,p_n)).
\end{equation}
To justify \eqref{eq: tostada}, first let $F_n$ be the event that there is an edge $e \in \tilde B(n): = B(n/4) + (3n/2)e_1$ such that the following conditions hold:
\begin{enumerate}
\item the weight $t_e \in (p_c,p_n)$,
\item one endpoint of the dual edge $e^*$ is connected to the top of the box $[n,2n]\times [-2n,2n]$ by a $p_n$-closed dual path (inside the box), and the other is connected to the bottom by another disjoint such path,
\item one endpoint of $e$ is connected to $\partial B(n/2)$ by a $p_c$-open path remaining in the set $[0,2n] \times [-n,n]$, and
\item the other endpoint of $e$ is connected to $\partial B(4n)$ by $p_c$-open path remaining in $[0,4n]\times [-n,n]$.
\end{enumerate}
By a similar argument to that given for Lemma~\ref{lma: outlet_construction}, but now explicitly directing the $p_c$-open arms as in \cite{kestenscaling}, one obtains
\[
\mathbb{P}(F_n) \geq C > 0.
\]

Next, we define the event $F_n'$ that there is a $p_n$-closed dual path $P$ connecting the top side of $[-2n,-n] \times [-2n,2n]$ to the bottom (within the box), and $P$ is connected in $[-2n,0] \times [-n,n]$ by $|\tilde \sigma|-1$ arms to $\partial B(n/2)$ so that these arms, along with the $p_c$-open one from the event $F_n$, have relative orientation $\tilde \sigma$. We claim that also
\begin{equation}\label{eq: clamland}
\mathbb{P}(F_n') \geq C > 0.
\end{equation}
To see why, we condition on the left-most $p_n$-closed dual path connecting the top of $[-2n,-n] \times [-2n,2n]$ to the bottom. For any such path $P$, let $D(P)$ be the event that it is leftmost. Then by the RSW theorem,
\[
0 \leq C \leq \sum_P \mathbb{P}(D(P)).
\]
Now let $F_n'(P)$ be the event that $P$ is connected by $|\tilde \sigma|-1$ arms to $\partial B(n/2)$ within $[-2n,0] \times [-n,n]$ so that these arms, along with the $p_c$-open one from the event $F_n$, have relative orientation $\tilde \sigma$. By the RSW theorem, $\mathbb{P}(F_n'(P)) > 0$ uniformly in $P$, and this event is independent of $D(P)$, so we obtain
\[
\mathbb{P}(F_n') \geq \sum_P \mathbb{P}(F_n'(P)) \mathbb{P}(D(P)) \geq C,
\]
proving \eqref{eq: clamland}. Note that by independence, we also have
\[
\mathbb{P}(F_n \cap F_n') \geq C.
\]

Next if $F_n''$ is the event that there are $p_n$-closed dual paths within $[-2n,2n]\times[n,2n]$ and $[-2n,2n]\times [-2n,-n]$ connecting the left and right sides of those boxes, then the RSW theorem gives $\mathbb{P}(F_n'') \geq C$ and by the generalized FKG inequality,
\[
\mathbb{P}(F_n \cap F_n' \cap F_n'') \geq C.
\]
Once again, we apply the generalized FKG inequality to obtain $\mathbb{P}(G_n) \geq C$, where $G_n$ is the event that $F_n \cap F_n' \cap F_n''$ occurs, and the $p_c$-open path from item 4 of the definition of $F_n$ is connected to infinity by a $p_n$-open path. (Here we are using \eqref{eq: p_n_to_infty}.) Therefore
\[
\mathbb{P}(A_{\tilde \sigma}(K,n,p_c,p_n)) \leq C \mathbb{P}(A_{\tilde \sigma}(K,n/2,p_c,p_n)) \mathbb{P}(G_n).
\]
By a gluing argument with the generalized FKG inequality,
\[
\mathbb{P}(A_{\tilde \sigma}(K,n/2,p_c,p_n)) \mathbb{P}(G_n) \leq \mathbb{P}(A_{\tilde \sigma}^*(K,n,p_c,p_n)),
\]
and these two inequalities justify \eqref{eq: tostada}.

We have now placed an outlet-like edge $e$ in the annulus $Ann(n,2n)$ (although it still needs to be guaranteed to be invaded -- this will occur later by forcing a $p_c$-open path from 0 to it), so we set on our next task, which is to make sure that the adjacent closed arms described in the sequence $\tilde \sigma$ can be chosen not to come too close to each other. Whenever $n \geq 8K \geq 8k|\tilde {\sigma}|$, where 
\[
k = |\sigma|,
\] 
we define $A_{\tilde {\sigma}}^{*,k}(K,n,p_c,p_n)$ that $A_{\tilde {\sigma}}^*(K,n,p_c,p_n)$ occurs, but the closed arms for the $\tilde \sigma$-connection can be chosen so that each pair associated to adjacent `$C$' entries is \emph{$k$-separated} in $Ann(2K,n/4)$: their intersections with $Ann(2K,n/4)$ remain Euclidean distance at least $k$ {from} each other. (Note that each pair of length-two sequences of `$C$' entries in $\tilde{\sigma}$ are separated by `$O$' entries.) In this modified event, we still insist that one of the $p_c$-open paths connects $B(K)$ to $e$ and the other $p_c$-open paths connect $B(K)$ to the dual circuit $\mathcal{C}_e$ from the definition of $O_e$.



\begin{claim}\label{clam: clam_head}
For any $K$ large enough, one has for $n \geq 8K$,
\begin{equation}\label{eqn: AlessAstar}
\mathbb{P}(A_{\tilde \sigma}^*(K,n,p_c,p_n))\le 2 \mathbb{P}(A_{\tilde \sigma}^{*,k}(K,n,p_c,p_n)).
\end{equation}
\end{claim}
\begin{proof}



To show this claim, the main step is to show (in the following lemma) that when the event on the left occurs, but the one on the right does not, we can find an edge associated to a certain ``six-arm'' event. See Figure \ref{arms_construction}.

\begin{figure}
\centering
\includegraphics[scale = 0.30]{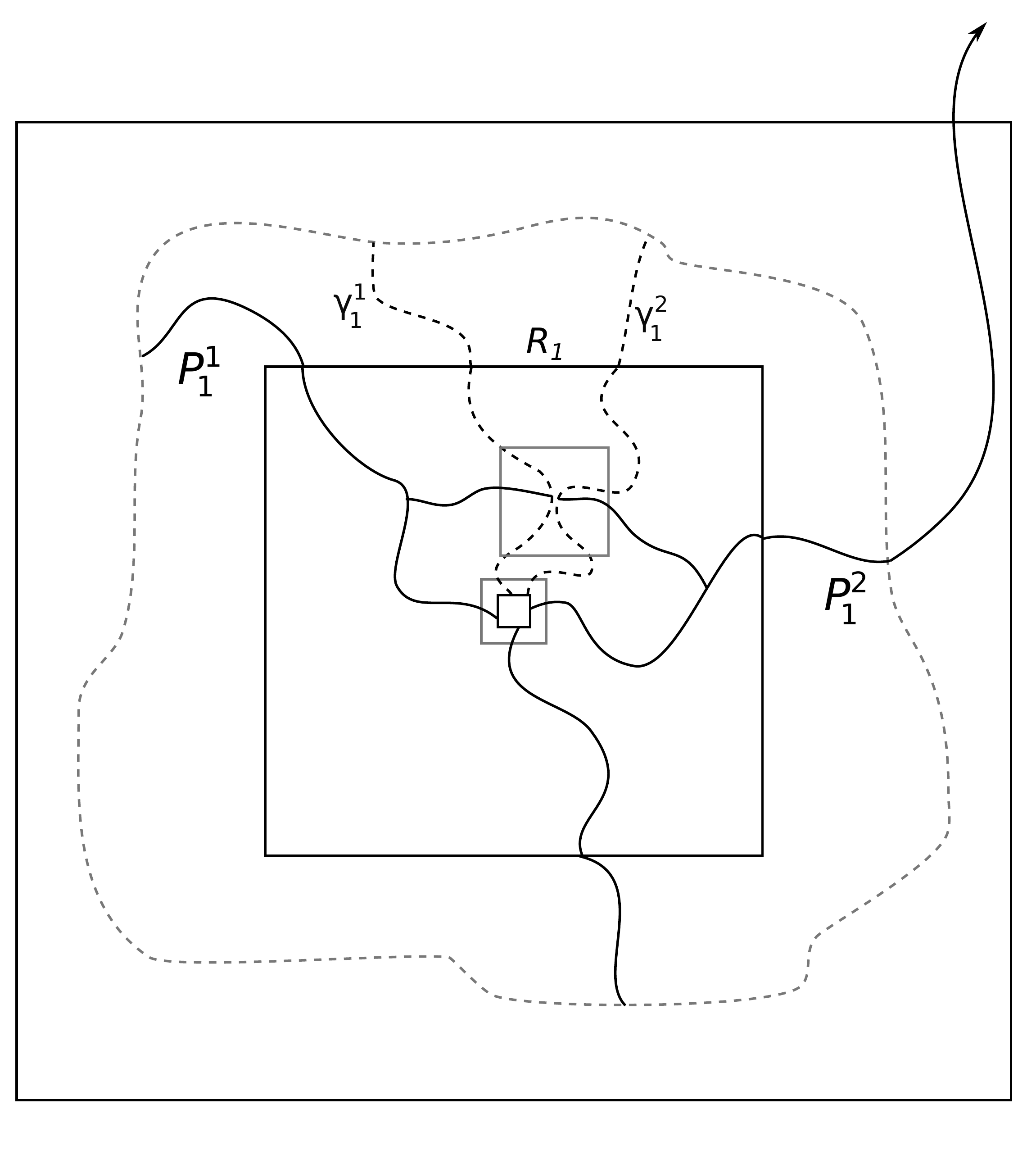}
\caption{The arms construction from Lemma \ref{lma: lma_1}. We show that the closed arms in $A_{\tilde \sigma}^*(K, n, p_c, p_n)$ may be chosen separated by arguing that otherwise there would be a six-arm point. In the figure, this six-arm point is where the paths $\gamma_1^1$, $\gamma_1^2$, and two open paths nearly meet; the arms extend from this meeting point to the boundary of the small box.}
\label{arms_construction}
\end{figure}

\begin{lma}\label{lma: lma_1}
If $n \geq 8K \geq 8 k|\tilde{\sigma}|$ and $A_{\tilde \sigma}^*(K,n,p_c,p_n)$ occurs but $A_{\tilde \sigma}^{*,k}(K,n,p_c,p_n)$ does not, then there is an edge $f\in Ann(2K,n/4)$ with the following six-arm conditions to distance $|f|/2$:
\begin{enumerate}
\item four disjoint $p_n$-closed dual arms from $B(f,2k)$ (the rectangular box of sidelength $2k$ centered on the midpoint of $f$) to $\partial B(f, |f|/2)$, and
\item two additional disjoint $p_n$-open arms from $B(f,2k)$ to $\partial B(f, |f|/2)$.
\end{enumerate}
The arms follow the color sequence $(OCCOCC)$.
\end{lma}

\begin{proof}[Proof of Lemma~\ref{lma: lma_1}]
Suppose that $A_{\tilde \sigma}(K,n, p_c,p_n)$ occurs, but $A_{\tilde \sigma}^k(K,n,p_c,p_n)$ does not. Fixing any choice of arms satisfying the event $A_{\tilde \sigma}(K,n,p_c,p_n)$, we will now make possibly new choices for the closed arms in each sequence of length 2 of `$C$' entries. If there are $r$ subsequences of length two of `$C$' entries, we first fix $p_n$-closed dual arms $\gamma_1^1, \gamma_2^1, \ldots, \gamma_1^r, \gamma_2^r$ (in clockwise order) corresponding to them in $int(\mathcal{C}_e)\setminus B(K)$ from $\partial B(K)$ to the $p_n$-closed dual circuit $\mathcal{C}_e$ with one defect. For each $i$, the arms $\gamma_i^1$ and $\gamma_i^2$ are bounded on either sides by (possibly non-distinct) $p_c$-open arms $P_i^1, P_i^2$, which delimit a region $R_i$ between $\partial B(K)$ and $\mathcal{C}_e$. We now choose $\hat \gamma_i^1$ and $\hat \gamma_i^2$ to be $p_n$-closed dual arms in $R_i$ which are counterclockwise-most and clockwise-most (so that the region between them is maximal). 

The arms $\hat \gamma_i^1, \hat \gamma_i^2$ for $i=1, \ldots r$, along with the $P_i^1,P_i^2$ give a total of $|\tilde \sigma|$ arms from $\partial B(K)$ to $\mathcal{C}_e$ which are oriented according to $\tilde \sigma$, so since $A_{\tilde \sigma}^{*,k}(K,n,p_c,p_n)$ does not occur, there must be an $i$ such that $\hat \gamma_i^1$ and $\hat \gamma_i^2$ come within distance $k$ of each other in $Ann(2K,n/4)$. Therefore we can find $f \in Ann(2K,n/4)$ such that the arms $\hat \gamma_i^1,\hat \gamma_i^2$ intersect $B(f,k)$. These arms furnish the four disjoint $p_n$-closed dual arms in part 1. By extremality, for each dual edge $g^*$ on $\hat \gamma_i^1\cup\hat \gamma_i^2$, there is a $p_n$-open path from $g$ to one of $P_i^1$ or $P_i^2$. Since both of $\hat \gamma_i^1$ and $\hat \gamma_i^2$ intersect $B(f,k)$, we find $p_n$-open arms which can be continued to $\partial B(f,|f|/2)$ when they meet either $P_i^1$ or $P_i^2$.
\end{proof}

We continue with $n \geq 8K \geq 8 k|\tilde{\sigma}|$ as in Lemma~\ref{lma: lma_1}. For $f \in Ann(2K,n/4)$, write $E_f$ for the event in the Lemma~\ref{lma: lma_1}. Then
\begin{equation}\label{eq: taco_also}
\mathbb{P}(A_{\tilde \sigma}^*(K,n,p_c,p_n) \setminus A_{\tilde \sigma}^{*,k}(K,n,p_c,p_n)) \leq \sum_{f \in Ann(2K,n/4)} \mathbb{P}(A_{\tilde \sigma}^*(K,n,p_c,p_n),~E_f).
\end{equation}
Using independence and a gluing argument (similar to \cite[p.~1599]{nolin}), the above summand is at most
\begin{equation}\label{eq: taco_too}
\mathbb{P}(A_{\tilde \sigma}(K,|f|/2,p_c,p_n))\mathbb{P}(E_f) \mathbb{P}(A_{\tilde \sigma}^*(2|f|,n,p_c,p_n)) \leq C \mathbb{P}(A_{\tilde \sigma}^*(K,n,p_c,p_n)) \mathbb{P}(E_f).
\end{equation}
By \eqref{eq: changearms},
\[
\mathbb{P}(E_f) \leq C \mathbb{P}_{cr}(A_\rho(|f|/2)).
\]
Here, $\rho$ is the sequence $(OCCOCC)$, and $A_\rho(|f|/2)$ is the $\rho$-connection event between $\partial B(6)$ and $\partial B(|f|/2)$. (The constant $C$ depends on $k$, but since $k$ is fixed for us, this does not matter.) By Reimer's inequality \cite{reimer} and the RSW theorem, there is $\delta>0$ such that $\mathbb{P}_{cr}(A_\rho(|f|/2)) \leq C|f|^{-\delta}\mathbb{P}_{cr}(A_{\rho'}(|f|/2))$, where $\rho' = (OOCOC)$. The universal value of the 5-arm exponent is 2 \cite[Lemma 5]{KSZ}, \cite[Theorem 24, 3]{nolin}, so we obtain
\[
\mathbb{P}(E_f) \leq C |f|^{-2-\delta}.
\]

Combine this inequality on $\mathbb{P}(E_f)$ with \eqref{eq: taco_also} and \eqref{eq: taco_too} for
\begin{align*}
\mathbb{P}(A_{\tilde \sigma}^*(K,n,p_c,p_n) \setminus A_{\tilde \sigma}^{*,k}(K,n,p_c,p_n)) &\leq C \mathbb{P}(A_{\tilde \sigma}^*(K,n,p_c,p_n)) \sum_{f \in Ann(2K,n/4)} |f|^{-2-\delta} \\
&\leq CK^{-\delta} \mathbb{P}(A_{\tilde \sigma}^*(K,n,p_c,p_n)).
\end{align*}
Choosing any $K \geq k|\tilde{\sigma}|$ so that $CK^{-\delta} \leq 1/2$ gives
\begin{align*}
\mathbb{P}(A_{\tilde \sigma}^{*,k}(K,n,p_c,p_n)) &= \mathbb{P}(A_{\tilde \sigma}^*(K,n,p_c,p_n)) - \mathbb{P}(A_{\tilde \sigma}^*(K,n,p_c,p_n) \setminus A_{\tilde \sigma}^{*,k}(K,n,p_c,p_n)) \\
& \geq (1/2) \mathbb{P}(A_{\tilde \sigma}^*(K,n,p_c,p_n))
\end{align*}
and completes the proof of \eqref{eqn: AlessAstar}.
\end{proof}

From now on, fix $K \geq k |\tilde \sigma|$ as in {the proof of} Claim~\ref{clam: clam_head}. Combine \eqref{eq: changearms}, \eqref{eq: tostada}, and \eqref{eqn: AlessAstar} for $n \geq 8K$ to get
\begin{equation}\label{eq: last_banana}
\mathbb{P}_{cr} (A_{\tilde{\sigma}}(K,n)) \leq C \mathbb{P}(A_{\tilde \sigma}^{*,k}(K,n,p_c,p_n)).
\end{equation}
(From here on, we fix $K$.) In this last step, we must force the edge $e$ for which $O_e$ occurs to actually be an outlet, and force the $p_c$-open arms in the $\tilde \sigma$-connection to be invaded. To do this, we note that $A_{\tilde \sigma}^{*,k}(K,n,p_c,p_n)$ implies the event $\hat A_{\tilde \sigma}(2K,n,p_c,p_n)$ that all of the conditions from $A_{\tilde \sigma}^{*,k}(K,n,p_c,p_n)$ occur, but that the $\tilde \sigma$ arms begin at $\partial B(2K)$ (instead of $\partial B(K)$) and the pairs of adjacent closed arms remain $k$-separated in $Ann(2K,n/4)$. Because this event only depends on the state of edges in $B(2K)^c$, we will combine it with a fixed event defined in terms of edges in $B(2K)$ to ensure the additional necessary invasion conditions.

Since there are only finitely many choices of starting points on $\partial B(2K)$ for our arms in $\hat A_{\tilde \sigma}(2K,n,p_c,p_n)$, we can choose vertices $\{x_i\}$, dual vertices $\{y_i\}$, and a constant $C$ such that, conditioned on $\hat A_{\tilde \sigma}(2K,n,p_c,p_n)$ the probability is at least $C$ that $\{x_i\}$ and $\{y_i\}$ are the starting points of $\tilde \sigma$-arms in the definition of $\hat A_{\tilde \sigma}(2K,n,p_c,p_n)$. We simply want to connect these points to arms from the origin. So we consider the event $\hat E$ that the following conditions hold. See Figure \ref{insideB2K} for an illustration.
\begin{enumerate}
\item There are $|\tilde \sigma|_O$ $p_c$-open paths from 0 to the $x_i$'s. They are disjoint outside of $B(|\tilde \sigma|)$.
\item There are $|\tilde \sigma|_C$ disjoint $p_n$-closed dual paths from $\partial B(|\tilde \sigma|)$ to the $y_i$'s.
\item There is a $p_n$-closed dual circuit around 0 in $Ann(|\tilde \sigma|,|\sigma|)$ with $|\tilde \sigma|_O$ defects.
\item There are $|\sigma|$ disjoint paths from $\partial B(|\sigma|)$ to $\partial B(2K)$. $|\sigma|_O$ of these are portions of the $p_c$-open paths from item 1, $|\tilde \sigma|_C$ of these are portions of the $p_n$-closed dual paths from item 2, and $|\sigma|_C - |\tilde \sigma|_C$ of these have no weight restriction, and lie between the adjacent closed dual paths from item 2 in such a way that if we were to count these additional dual paths as closed, then there would be a $\sigma$-connection from $\partial B(|\sigma|)$ to $\partial B(2K)$.
\end{enumerate} 
In item 4, the non-weight restricted paths lie between pairs of closed dual paths in $\tilde \sigma$ and function as the additional closed arms (in the IPC) forming the arm event with sequence $\sigma$.

\begin{figure}
\centering
\includegraphics[scale = 0.35]{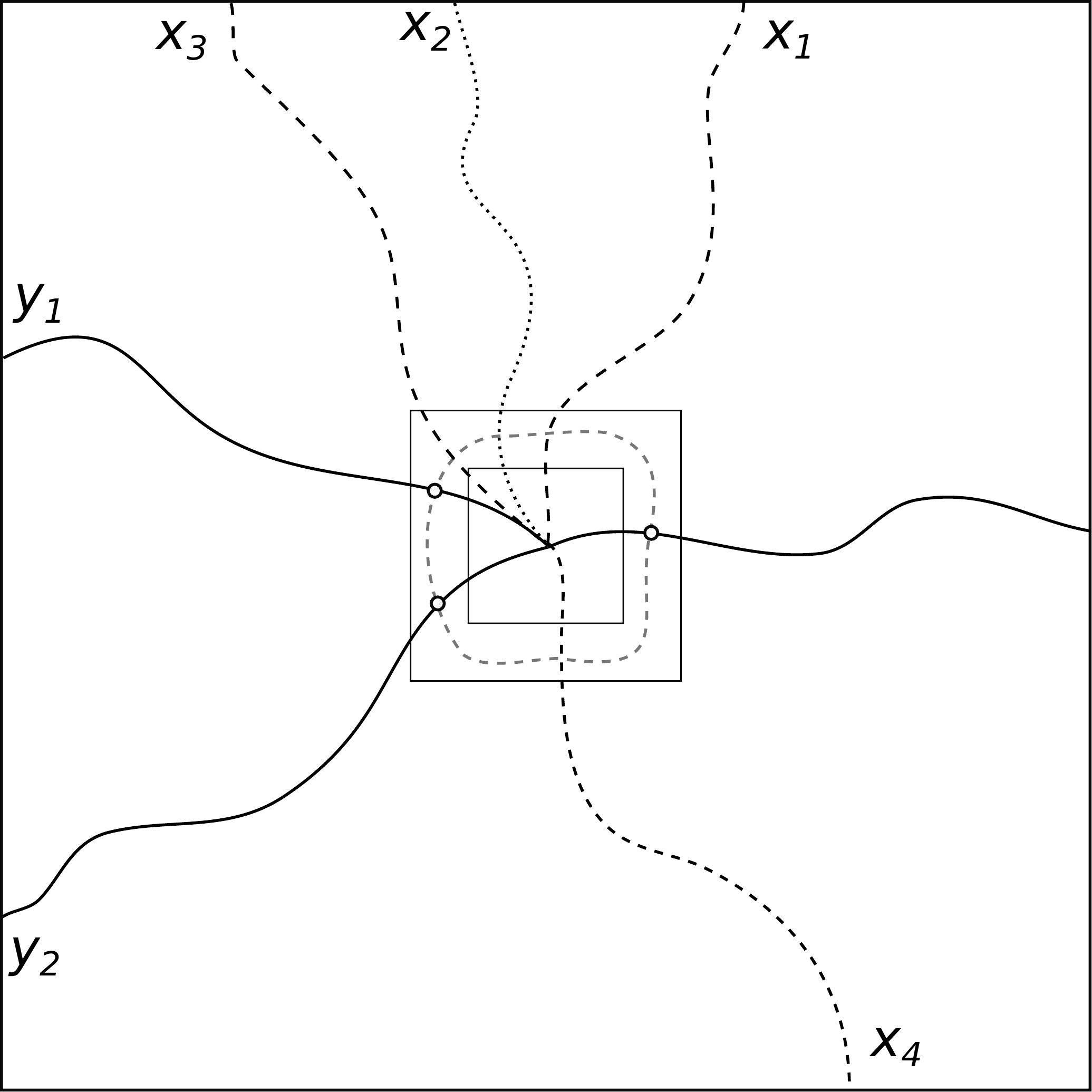}
\caption{The construction inside $B(2K)$ used to build the ``fixed event'' $\hat E$. When this event is intersected with $\hat A_{\tilde \sigma}$ and the event $\hat F$ (which fixes the starting points of the arms in $\hat A_{\tilde \sigma}$), the invasion will have arms as in the definition of $A_\sigma$. The boxes, in order from smallest to largest, are $B(|\tilde{\sigma}|)$, $B(|\sigma|)$, $B(2K)$.} 
\label{insideB2K}
\end{figure}

Calling $\hat F$ the event that our fixed $\{x_i\}$ and $\{y_i\}$ are the starting points on $\partial B(2K)$ of arms in $\hat A_{\tilde \sigma}(2K,n,p_c,p_n)$, one then has by independence
\begin{equation}\label{eq: first_banana}
\mathbb{P}(A_{\tilde \sigma}^{*,k}(K,n,p_c,p_n)) \leq C \mathbb{P}(\hat A_{\tilde \sigma}(2K,n,p_c,p_n), \hat F, \hat E).
\end{equation}
On the event on the right, all $p_c$-open arms from 0 in the definition of $\hat A_{\tilde \sigma}(2K,n,p_c,p_n) \cap \hat E$ must be invaded, and the edge $e$ from the definition of $O_e$ is also invaded -- it is an outlet. For this reason, no $p_n$-open edges are ever invaded, and so the $p_n$-closed dual arms from the definition of $\hat A_{\tilde \sigma}(2K,n,p_c,p_n) \cap \hat E$ are not invaded. Due to the presence of the $p_n$-closed dual circuit with defects from item 3, no edges are invaded in the regions delimited by (a) the adjacent $p_n$-closed dual arms from item 2, (b) the $p_n$-closed dual circuit from item 3, and (c) the closed dual circuit $\mathcal{C}_e$ with one defect from the definition of $O_e$. Therefore the additional paths from item 4 are not invaded. Since the $p_n$-closed dual paths from $\hat A_{\tilde \sigma}(2K,n,p_c,p_n)$ remain at least distance $k$ from each other in $Ann(2K,n/4)$, the additional paths from item 4 can be extended to ones which reach $\partial B(n/4)$, and they will not be invaded. Therefore
\[
\mathbb{P}(\hat A_{\tilde \sigma}(2K,n,p_c,p_n), \hat F, \hat E) \leq \mathbb{P}(A_\sigma(n/4)).
\]
Combine this with \eqref{eq: last_banana} and \eqref{eq: first_banana} to obtain
\[
\mathbb{P}_{cr}(A_{\tilde \sigma}(K,n)) \leq C \mathbb{P}(A_{\sigma}(n/4)).
\]
By quasimultiplicativity of arm events, we finish with our claimed statement \eqref{eq: one_bound}:
\[
\mathbb{P}_{cr}(A_{\tilde \sigma}(n/4)) \leq C \mathbb{P}_{cr}(A_{\tilde \sigma}(K,n)) \leq C \mathbb{P}(A_{\sigma}(n/4)).
\]


\end{proof}

\subsection{Proof of the upper bound}\label{sec: upper-bd}

The next proposition is the upper bound complementing \eqref{eq: one_bound}. We let $\sigma$ be a color sequence with $|\sigma|_O \geq 2$. We must show that there is a $C>0$ such that for all $n$,
\begin{equation}\label{eq: two_bound}
\mathbb{P}(A_\sigma(n))\le C\mathbb{P}_{cr}(A_{\tilde{\sigma}}(n)).
\end{equation}

\begin{proof}
We first split the probability according to the location of the ``first outlet'' $\hat e_1$. This is the unique edge with $t_{\hat e_1} > t_e$ for all invaded edges $e$.
\begin{equation}\label{eqn: outletloc}
\mathbb{P}(A_{\sigma}(n), \hat{e}_1 \in B(n/2)) + \mathbb{P}(A_{\sigma}(n), \hat{e}_1\notin B(n/2)).
\end{equation}
To deal with the first probability, we use a lemma. We write $\hat \tau_1$ for the weight $t_{\hat e_1}$. We will need the following definitions: for $n \geq 1$, set $\log^{(0)}n = n$ and, for $n$ for which it is defined, set $\log^{(k)}n = \log (\log^{(k-1)}n)$. For $n > 10$, put
\[
\log^* n = \max\{ k \geq 0 : \log^{(k)} n > 10\},
\]
and for $n > 10$ and ${j = 1}, \ldots, \log^* n$ with $M>0$ fixed until later, define
\[
p_n(j) = p_{\left\lfloor \frac{n}{M \log^{(j)} n} \right\rfloor}.
\]
(Recall the definition of $p_n$ in \eqref{eq: p_n_def}.) 

\begin{figure}
\centering
\includegraphics[scale = 0.35]{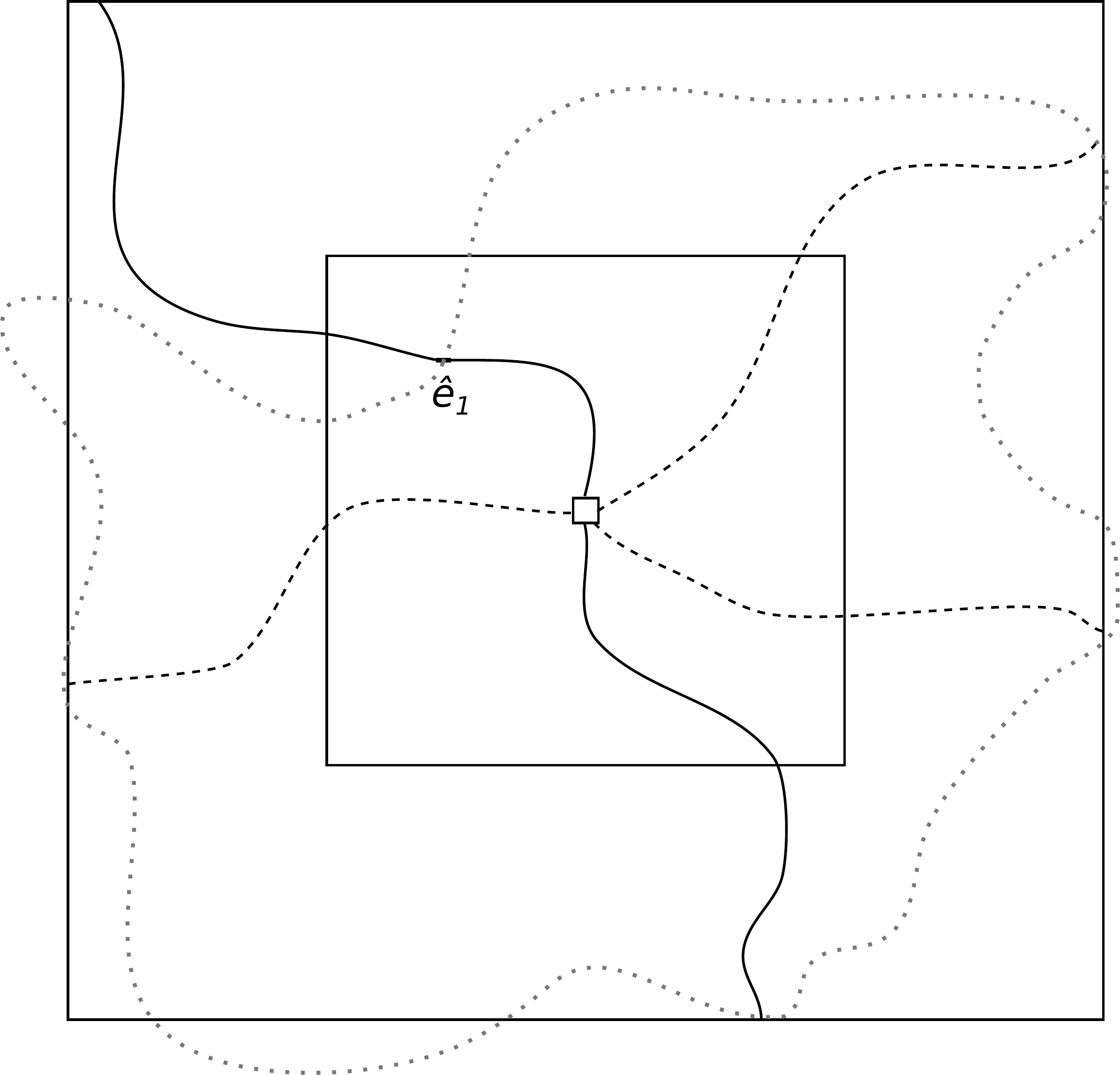}
\caption{An illustration of  Lemma \ref{lma: sigmareduction}. The first outlet $\hat{e}_1$ lies on an open arm. All open arms extend to $\partial B(n)$ (the outer box), so if $\hat{e}_1\in B(n/2)$ (the middle box) and has weight less than $\tau$, there are $\tau$-closed arms from $\hat{e}_1$ to the the extremities of the open arms. These have diameter {at least} $n/2$.}
\label{outlet_two_open}
\end{figure}

\begin{lma}\label{lma: sigmareduction}
Suppose that $A_\sigma(n)$ occurs. If $\hat{e}_1\in B(n/2)$ and $\hat{\tau}_1\le p_{n/2}(j)$, then the following occur.
\begin{enumerate} 
\item If $p_{n/2}(j+1)<\hat \tau_1$, there are two disjoint $p_{n/2}(j+1)$-closed dual arms from $\partial B(n/2)$ to $\partial B(n)$.
\item There are $|\sigma|_O$ disjoint $p_{n/2}(j)$-open arms from $\partial B(|\sigma|)$ to $\partial B(n/2)$.
\item There are $|\tilde{\sigma}|_C$ disjoint $p_c$-closed dual arms from $\partial B(|\sigma|)$ to $\partial B(n/2)$.
\item The arms from items 2 and 3 appear in the sequence specified by $\tilde{\sigma}$.
\end{enumerate}
\end{lma}
\begin{proof}

Let us consider these properties in order. See Figure \ref{outlet_two_open} for an illustration.
\begin{enumerate}
\item Just before adding $\hat{e}_1$ to the invasion, the entire invasion is $p_{n/2}(j)$-open and is surrounded by a $p_{n/2}(j+1)$-closed dual circuit around the origin containing $\hat{e}_1^*$. Only one edge, $\hat e_1$, on this closed dual circuit will ever be invaded. {Indeed, if $e$ is another edge on this circuit, we have $t_{e} > t_{\hat e_1}$, since $\hat e_1$ is invaded rather than $e$. On the other hand, $\hat e_1$ has the maximal weight of all invaded edges, which precludes the invasion of $e$.} {Thus,} $\hat e_1$ is an edge whose removal from the invasion disconnects 0 from infinity. Therefore $\hat e_1$ is on exactly one of the open arms from $\partial B(|\sigma|)$ to $\partial B(n)$, and any other open arm must be completely contained in the invasion at that point. To encompass this arm, the $p_{n/2}(j+1)$-closed dual circuit must cross the annulus $Ann(n/2,n)$, and so it must contain two disjoint $p_{n/2}(j+1)$-closed dual arms from $\partial B(n/2)$ to $\partial B(n)$.
\item  Since the first outlet has weight at most $p_{n/2}(j)$, so does every edge in the invasion, including $|\sigma|_O$ open arms from $\partial B(|\sigma|)$ to $\partial B(n/2)$.
\item Make a choice of arms from $\partial B(|\sigma|)$ to $\partial B(n)$ that satisfy the event $A_\sigma(n)$. Fix two open arms $\alpha_1$, $\alpha_2$ in the invasion which appear consecutively in the sequence $\sigma$, but are separated by at least $l \geq 1$ closed dual arms, and let $R$ be the region bounded by the $\alpha_i$'s, $\partial B(|\sigma|)$, and $\partial B(n/2)$. 

If $l=1$, since there is a non-invaded arm separating the two invaded arms, any path in $R$ joining $\alpha_1$ to $\alpha_2$ must contain at least one non-invaded edge. Since any $p_c$-open path joining $\alpha_1$ to $\alpha_2$ in $R$ would be contained in the invasion, there can be no such path in $R$. By duality there is at least one $p_c$-closed dual arm from {{the}} $\partial B(|\sigma|)$ to $\partial B(n/2)$ in $R$.

In the case $l \ge 2$, consider the two extremal non-invaded dual arms in $R$. That is,  $\beta_1$ is a non-invaded dual arm such that the region $R_1$ bounded by $\alpha_1$, $\beta_1$, $\partial B(|\sigma|)$, and $\partial B(n/2)$ is minimal. The other extremal arm $\beta_2$, similarly defined, serves as part of the boundary of a region $R_2$ whose boundary also includes $\alpha_2$. $\beta_1$ and $\beta_2$ are disjoint since $l\ge 2$, and so are $R_1$ and $R_2$. No $p_c$-open path in $R$ whose initial edge touches $\alpha_1$ can exit $R_1$ since such an arc must be invaded. It follows that $R_1\cup \beta_1$ contains a $p_c$-closed dual arm from $\partial B(|\sigma|)$ to $\partial B(n/2)$. The same holds for $R_2$, so there are at least two $p_c$-closed dual arms from $\partial B(|\sigma|)$ to $\partial B(n/2)$ in $R$.

The previous argument applies to any region between two consecutive open arms, and produces a total of $|\tilde{\sigma}|_C$ $p_c$-closed dual arms from $\partial B(|\sigma|)$ to $\partial B(n/2)$.

\item This follows immediately from the way the closed dual arms are obtained in the previous item.
\end{enumerate}
\end{proof}

Given the previous lemma, we now decompose the first probability in \eqref{eqn: outletloc} according to the value of $\hat \tau_1$. In case $\hat{\tau}_1>p_{n/2}(1)$ we bound
\[
\mathbb{P}(A_\sigma(n), \hat e_1 \in B(n/2), \hat \tau_1 > p_{n/2}(1))\le \mathbb{P}(C_n(1)),
\]
where $C_n(j)$ is the event that there are two disjoint $p_{n/2}(j)$-closed arms from $\partial B(n/2)$ to $\partial B(n)$. From \eqref{eq: closed_exponential} and \eqref{eq: D_p_n}, one obtains the upper bound
\[
\mathbb{P}(C_n(1)) \leq C_1 e^{-{2 C_2} \frac{n/2}{L(p_{n/2}(1))}} \leq C_1 e^{-\frac{2C_2}{D} M \log (n/2)}.
\]
By quasimultiplicativity, there is $\alpha>0$ such that for all $n \geq 2$,
\[
\mathbb{P}_{cr}(A_{\tilde \sigma}(n)) \geq Cn^{-\alpha},
\]
so if we choose $M$ large enough,
\begin{equation}\label{eq: first_decompose_term}
\mathbb{P}(A_\sigma(n), \hat e_1 \in B(n/2), \hat \tau_1 > p_{n/2}(1)) \leq \mathbb{P}_{cr}(A_{\tilde \sigma}(n))
\end{equation}
{for all $n$ sufficiently large.}

For the probability of the remaining event, $\{A_{\sigma}(n), \hat e_1 \in B(n/2), \hat{\tau}_1\le p_{n/2}(1)\}$, Lemma \ref{lma: sigmareduction} gives the bound
\begin{align}
&\mathbb{P}(A_\sigma(n), \hat e_1 \in B(n/2), \hat \tau_1 \leq p_{n/2}(1)) \nonumber \\
=~& \sum_{j=1}^{\log^*(n/2)-1} \mathbb{P}(A_\sigma(n), \hat e_1 \in B(n/2), \hat \tau_1 \in (p_{n/2}(j+1), p_{n/2}(j)]) \nonumber \\
+~& \mathbb{P}(A_\sigma(n), \hat e_1 \in B(n/2), \hat \tau_1 \leq p_{n/2}(\log^*(n/2))) \nonumber \\
\leq~& \sum_{j=1}^{\log^*(n/2)-1} \mathbb{P}(A_{\tilde \sigma}(|\sigma|, n/2, p_{n/2}(j), p_c), C_n(j+1)) \nonumber \\
+~& \mathbb{P}(A_{\tilde \sigma}(|\sigma|,n/2, p_{n/2}(\log^*(n/2)), p_c)). \label{eqn: sigmareduced-sum}
\end{align}
In this equation, we are using the events $A_{\tilde \sigma}(m,n,p,q)$ defined near \eqref{eq: changearms}.



The events $C_n(j+1)$ and  $A_{\tilde{\sigma}}(|\sigma|,n/2, p_{n/2}(j), p_c)$ depend on disjoint sets of edges, so by \eqref{eq: closed_exponential} and \eqref{eq: D_p_n}, the sum \eqref{eqn: sigmareduced-sum} is bounded by
\begin{align*}
&\sum_{j=1}^{\log^*(n/2)-1} C_1\mathbb{P}(A_{\tilde{\sigma}}(|\sigma|, n/2, p_{n/2}(j), p_c))\exp\left(-\frac{2C_2}{D} M \log^{(j+1)}(n/2) \right) \\
+~&\mathbb{P}( A_{\tilde{\sigma}}(|\sigma|,n/2,p_{n/2}(\log^*(n/2)), p_c)).
\end{align*}
By an explicit construction, we can change the term $|\sigma|$ to $|\tilde \sigma|$ at the cost of a constant factor, so we obtain
\begin{align}
&\sum_{j=1}^{\log^*(n/2)-1} C\mathbb{P}(A_{\tilde{\sigma}}(n/2, p_{n/2}(j), p_c))\exp\left(-\frac{2C_2}{D} M \log^{(j+1)}(n/2) \right) \nonumber \\
+~&C\mathbb{P}( A_{\tilde{\sigma}}(n/2,p_{n/2}(\log^*(n/2)), p_c)). \label{eq: burrito_suprema}
\end{align}
For the probability $\mathbb{P}(A_{\tilde \sigma}(n/2, p_{n/2}(j), p_c))$, we use quasimultiplicativity and \eqref{eq: changearms} for the upper bound
\begin{align*}
\mathbb{P}(A_{\tilde \sigma}(n/2,p_{n/2}(j), p_c)) &\leq C\mathbb{P}_{cr}\left(A_{\tilde \sigma}\left(\left\lfloor \frac{n/2}{M \log^{(j)}(n/2)} \right\rfloor \right) \right) \\
&\leq C \mathbb{P}_{cr}(A_{\tilde \sigma}(n)) \mathbb{P}_{cr}\left( A_{\tilde \sigma} \left( \left\lfloor \frac{n/2}{M \log^{(j)}(n/2)} \right\rfloor, n \right) \right)^{-1}.
\end{align*}
Once again, by quasimultiplicativity, we get the bound
\begin{equation}\label{eq: arm_change_bound}
\mathbb{P}(A_{\tilde \sigma}(n/2, p_{n/2}(j), p_c)) \leq C \mathbb{P}_{cr}(A_{\tilde \sigma}(n)) \left( M \log^{(j)} n \right)^{C_3}.
\end{equation}
Plug this into \eqref{eq: burrito_suprema} for the bound
\begin{equation}\label{eq: andre_on_holiday_sum}
C {M^{C_3}} \mathbb{P}_{cr}(A_{\tilde \sigma}(n)) \left( 1+ \sum_{j=1}^{\log^*(n/2)-1}  (\log^{(j)} (n/2))^{{C_3 - C_4  M} }\right).
\end{equation}
If $M$ is large enough, {$C_3 - C_4 M \leq -1$} and as in \cite[Eq.~(2.26)]{jarai}, the sum is bounded by a constant independently of $n$. So we finish with an upper bound for \eqref{eqn: sigmareduced-sum} of $C {\mathbb{P}_{cr}}(A_{\tilde \sigma}(n/2))$. Combining this with \eqref{eq: first_decompose_term}, we complete the bound on the first term of \eqref{eqn: outletloc}:
\begin{equation}\label{eq: first_outletloc_final_bound}
\mathbb{P}(A_\sigma(n), \hat e_1 \in B(n/2)) \leq C {\mathbb{P}_{cr}}(A_{\tilde \sigma}(n)).
\end{equation}


We now aim to give a similar bound for the second term in \eqref{eqn: outletloc}. Once again we need a lemma about arm events, but with another definition. We write
\begin{equation}\label{eqn: Hndef}
H_n(j) = \left\{
\begin{array}{c}
\text{there is a } p_n(j)\text{-open circuit } \mathcal{D} \text{ around } 0 \text{ in }  B(n)\setminus B( n/2 )   \\
 \text{ and } \mathcal{D}\leftrightarrow \infty \text{ by a }p_n(j)\text{-open path in } B(n/2)^c 
 \end{array}\right\}.
\end{equation}
As in \cite[Eq. (2.21)]{jarai}, we have, for some $C,c>0$
\[
\mathbb{P}(H_n(j)^c) \le C\exp(-cM\log^{(j)} n) \text{ for all } n,j.
\]

\begin{figure}
\centering
\includegraphics[scale = 0.35]{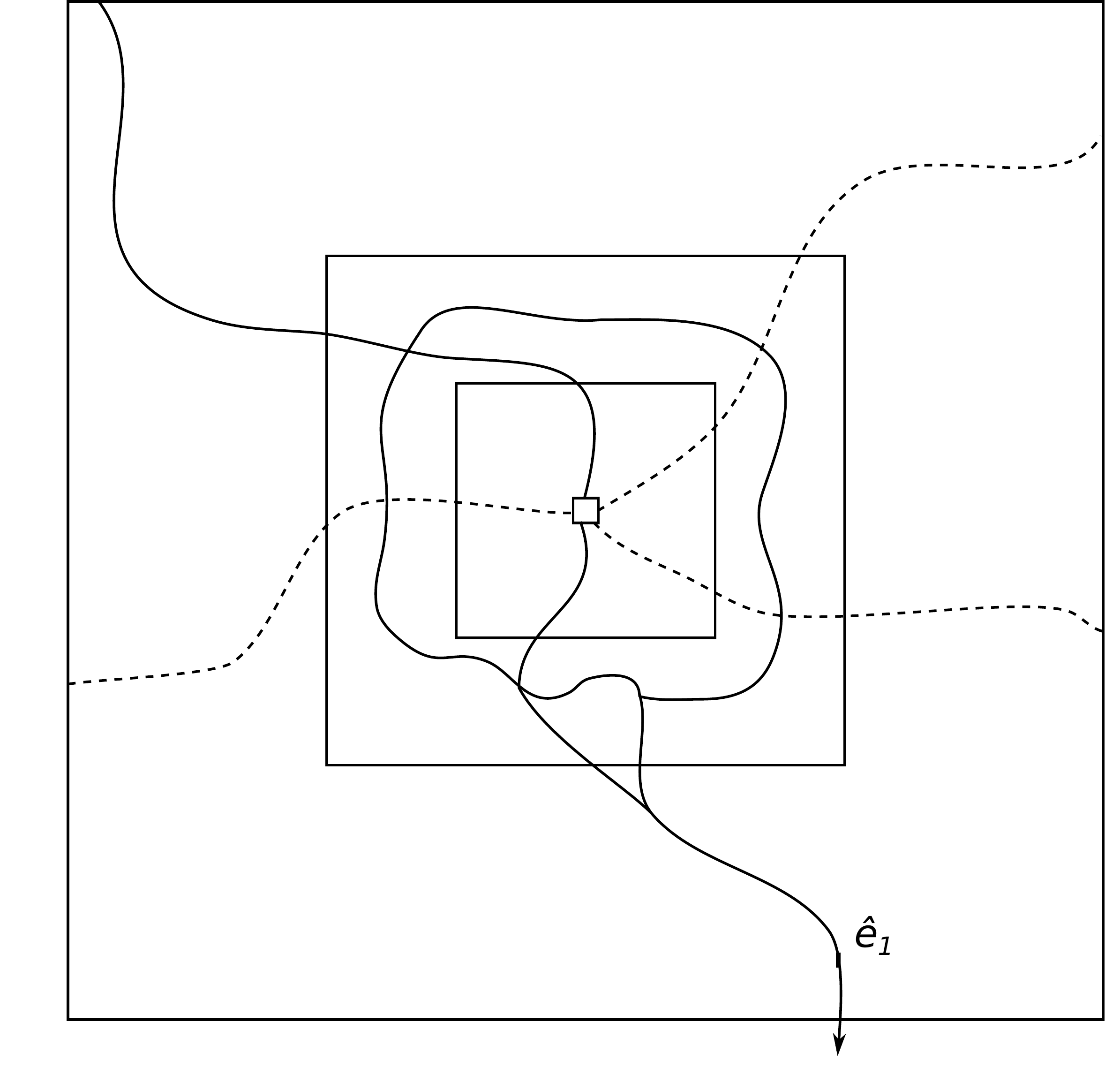}
\caption{An illustration of  Lemma \ref{lma: sigmareduction2}. The occurrence of a $p_{n/2}(j)$ open circuit in $Ann(n/4,n/2)$ $p_{n/2}(j)$-connected to $\infty$ implies that $\tau_{\hat{e}_1}\le p_{n/2}(j)$ if $\hat{e}_1\in B(n/2)^c$. By definition, all invaded arms are $p_{n/2}(j)$-open in this case. The boxes, in order from smallest to largest, are $B(|\sigma|)$, $B(n/4)$, $B(n/2)$, $B(n)$.}
\label{outlet_outside}
\end{figure}

\begin{lma}\label{lma: sigmareduction2}
Suppose $A_\sigma(n)$ occurs. If $\hat{e}_1\notin B(n/2)$ and $H_{n/2}(j)$ occurs, then the following hold.
\begin{enumerate} 
\item There are $|\sigma|_O$ $p_{n/2}(j)$-open arms from $\partial B(|\sigma|)$ to $\partial B(n/4)$.
\item There are $|\tilde \sigma|_C$ $p_c$-closed dual arms from $\partial B(|\sigma|)$ to $\partial B(n/4)$.
\item The arms appear in the sequence specified by $\tilde{\sigma}$.
\end{enumerate}
\end{lma}
\begin{proof}
The proof is essentially identical to that of Lemma \ref{lma: sigmareduction}. The only difference is that it is now the occurence of $H_{n/2}(j)$ which forces $\hat{\tau}_1\le p_{n/2}(j)$.
\end{proof}
The occurence of $H_{n/2}(j)$ depends only on edges outside $B(n/4)$, so can write:
\begin{align*}
 &\mathbb{P}(A_\sigma (n), \hat{e}_1 \in B^c( n/2 ))\\
\leq~&\sum_{j=1}^{\log^*( n/2  )-1} \mathbb{P}(A_{\tilde{\sigma}}(|\sigma|,n/4,p_{n/2}(j),p_c), \hat{e}_1 \in B^c( n/2 ), H_{ n/2 }(j+1)^c, H_{ n/2  }(j)) \\
+~& \mathbb{P}(A_{\tilde{\sigma}}(|\sigma|,n/4,p_{n/2}(\log^*(n/2)),p_c), \hat{e}_1 \in B^c( n/2 ),H_{ n/2  }(\log^*( n/2 ))) + {\mathbb{P}(H_{n/2}(1)^c)}\\
\le~& C\sum_{j=1}^{\log^*( n/2 )-1} \mathbb{P}(A_{\tilde{\sigma}}(n/4,p_{n/2}(j),p_c))\mathbb{P}(H_{n/2}(j+1)^c)\\
~&+ C\mathbb{P}(A_{\tilde{\sigma}}(n/4, p_{n/2}(\log^*(n/2)),p_c)) + {\mathbb{P}(H_{n/2}(1)^c)}.
\end{align*}

Since $\mathbb{P}(H_{n/2}(j+1)^c)\le C\exp(-cM \log^{(j+1)}(n/2))$, we conclude by a summation similar to \eqref{eq: andre_on_holiday_sum}, since, as in \eqref{eq: arm_change_bound}
\[
\mathbb{P}(A_{\tilde{\sigma}}(n/4,p_{n/2}(j),p_c)) \le  C\mathbb{P}_{cr}(A_{\tilde{\sigma}}(n)) (M\log^{(j)} n)^{C_4}.
\]
We then produce
\[
\mathbb{P}(A_\sigma(n), \hat e_1 \in B^c(n/2)) \leq C \mathbb{P}_{cr}(A_{\tilde \sigma}(n)),
\]
and combine this with \eqref{eq: first_outletloc_final_bound} for the final bound \eqref{eq: two_bound}.
\end{proof}

\section{Proof of Theorem~\ref{thm: only_one_a}}

\subsection{Lower bound for $k \geq 1$} \label{sec: thm2-lwr-bd}

Our next result is the first inequality in Theorem~\ref{thm: only_one_a}. It is a lower bound on the probability of $A_{\sigma_k}(n)$ in the IPC, and we state itagain here for the reader's convenience: for some $\epsilon>0$ independent of $k$ and some $C$ depending on $k$,
\begin{equation}\label{eq: to_prove_outlet_jam}
n^{\epsilon} \mathbb{P}_{cr}(A_{\tilde{\sigma}_k}(n)) \leq C\mathbb{P}(A_{\sigma_k}(n)) \text{ for all } n.
\end{equation}
Here, $\tilde \sigma_k$ is the reduced arm sequence defined in our case by replacing the $k$ `$C$' entries by $\min\{k,2\}$ of them.

\begin{proof}[Proof of \eqref{eq: to_prove_outlet_jam}.]
We will deal with the cases $k \leq 2$ and $k > 2$ separately; first we take $k\leq 2$ so that $\tilde{\sigma}_k = \sigma_k$. Our main goal will be to find a positive density (in scales) of outlets for the IPC on the event $A_{\sigma_k}(n)$. (Recall that $e \in S$ in an outlet if, when $e = e_s$, the $s$-th invaded edge, one has $t_{e_s} > \sup\{t_{e_r} : r > s\}$.) As usual, we will consider only $n$ of the form $n = 2^N$ for some $N > 0$. The reason is quasimultiplicativity: if the statement holds for such $n$, then for a general $n$, putting $\hat n = 2^{\lceil \log_2 n \rceil}$,
\[
n^\epsilon \mathbb{P}_{cr}(A_{\tilde \sigma_k}(n)) \leq C(\hat n)^\epsilon\mathbb{P}_{cr}(A_{\tilde \sigma_k}(\hat n)) \leq C \mathbb{P}(A_{\sigma_k}(\hat n)) \leq  {C}\mathbb{P}(A_{\sigma_k}(n)).
\]

Fix a large positive integer $\kappa$ to be chosen later; this will govern how often we look for outlets. Ultimately, $\kappa$ will not depend on $N$, but only on $N$-independent constants related to, e.g., gluing and arm separation. We also fix some notation to avoid proliferation of symbols. Let $i_0$ be any fixed number with the property $B(|\sigma_k|) \subset B(2^{2\kappa i_0})$ (so that there is enough space in $B(2^{2\kappa i_0})$ to fit $k+1$ disjoint arms). Let $N' = \lfloor N/2\kappa \rfloor - 1$, where we assume $N$ is so large that $N' > i_0$. For each $i = i_0, \ldots, N'$ and $0 \leq a < \kappa$, let $D_i = Ann(2^{2 \kappa i}, 2^{2 \kappa i + 2\kappa})$ and $E_i^a = Ann(2^{2 \kappa i + 2a}, 2^{2 \kappa i + 2a +1 })$. An arbitrary vector $\rho = (\rho(i_0), \, \ldots, \, \rho(N')) \in \{0, 1, \ldots, \kappa-1\}^{N'-i_0+1}$ will denote a particular choice of placement of outlets -- we will demand that exactly one outlet occur in each $D_i$, and that it be placed particularly in some $E_i^a$.  For a given $\rho$, we abbreviate the corresponding near-critical probabilities: $\widetilde p_i := p_{2^{2\kappa i+2\rho(i)}}$.

\begin{figure}
\centering
\includegraphics[scale = 0.35]{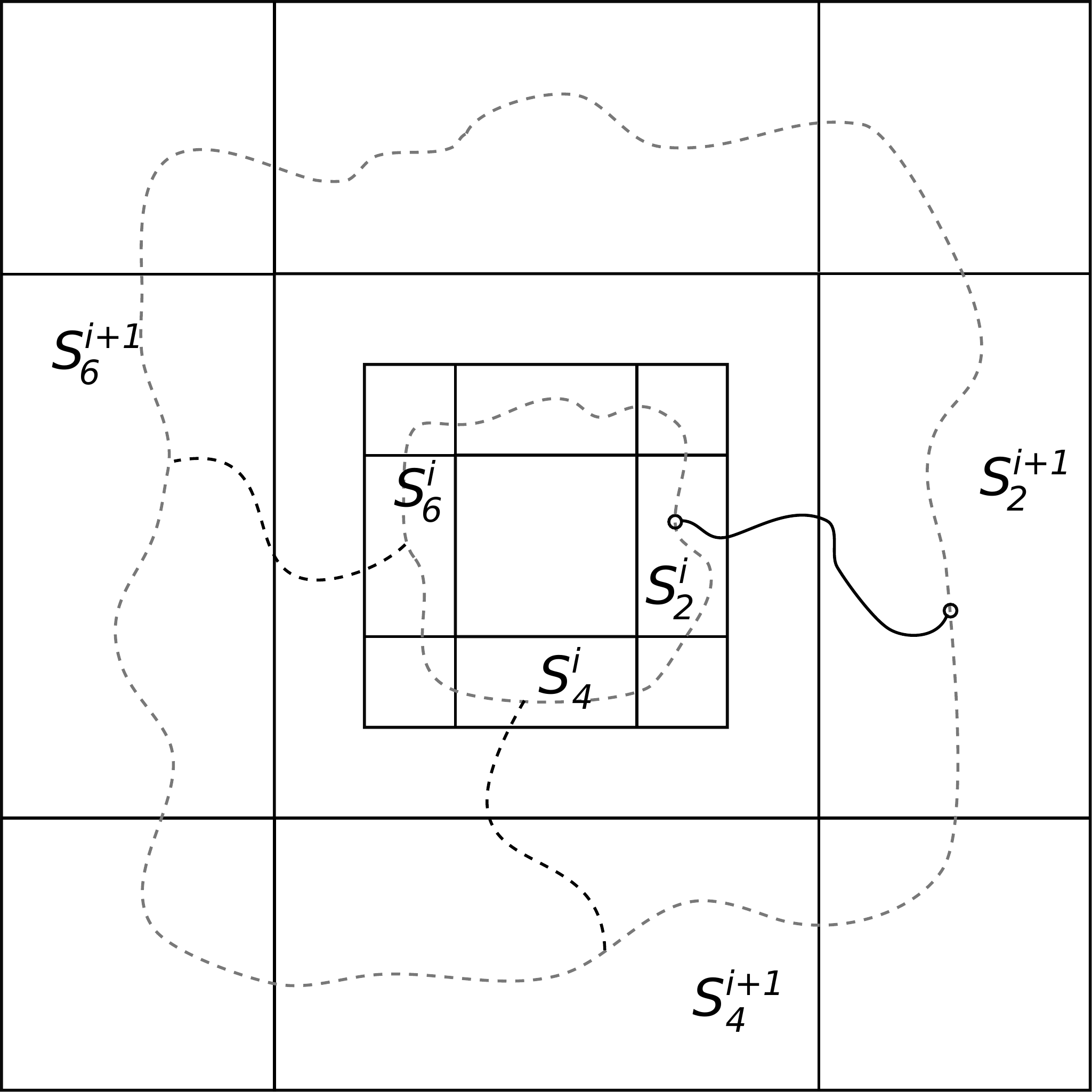}
\caption{An illustration of $E_i^{\rho(i)}$ (inner annulus) and $E_{i+1}^{\rho(i+1)}$ (outer annulus). A $p_c$-open path connects the defect in the $\tilde{p}_i$-closed dual circuit in $S_2^i$ to the defect in the $\tilde{p}_{i+1}$-closed dual circuit in $S_2^{i+1}$. In the case shown, $k=2$ $\tilde{p}_{i+1}$-closed dual paths connect the two circuits. Their intersections with the circuits lie in $S^i_4$, $S^{i+1}_4$ and $S^i_6$, $S_6^{i+1}$, respectively.}
\label{Ei}
\end{figure}

Formally, for each such $\rho$, define the event $K_\rho$ as follows. Split each annulus $E_i^{\rho(i)}$ into 8 rectangular boxes with disjoint interiors numbered $S_1^i, \ldots, S_8^i$ clockwise starting from the upper right corner, in the natural way. (That is, $S_1^i$ is the upper-right of the annulus, $S_2^i$ is the right, $S_3^i$ is the bottom-right, $S_4^i$ is the bottom, and so on.)
\begin{enumerate}
\item For each $i=i_0, \ldots, N'$, there is a $\widetilde p_i$-closed dual circuit with one defect in $E_i^{\rho(i)}$;
\item The aforementioned defected edge $e$ in each $E_i^{\rho(i)}$ is in $S_2^i$ and has $t_e \in (p_c, \widetilde p_i)$;
\item For each $i =i_0, \ldots, N'-1$, there are $k$ disjoint $\widetilde p_{i+1}$-closed dual paths connecting the above mentioned dual circuit in $E_i^{\rho(i)}$ to that in $E_{i+1}^{\rho(i+1)}$ so that their intersections with $E_i^{\rho(i)}$ and $E_{i+1}^{\rho(i+1)}$ are in $S_4^i$, $S_4^{i+1}$, and $S_6^i$, $S_6^{i+1}$, respectively, and the dual circuit in $E_{N'}^{\rho(N')}$ has $k$ disjoint $p_{2^N}$-closed dual connections to $\partial B(2^N)$ so that their intersections with $E_{N'}^{\rho(N')}$ are in $S_4^{N'}$ and $S_6^{N'}$, respectively (only $S_4$'s are used if $k=1$);
\item For each $i = i_0, \ldots, N' - 1$, there is a $p_c$-open path connecting the defects in $E_i^{\rho(i)}$ and $E_{i+1}^{\rho(i+1)}$ so that their intersections with $E_i^{\rho(i)}$ and $E_{i+1}^{\rho(i+1)}$ are in $S_2^i$ and $S_2^{i+1}$, respectively, and the defect in $E_{N'}^{\rho(N')}$ has a $p_c$-open connection to $\partial B(2^{N+1})$ so that its intersection with $E_{N'}^{\rho(N')}$ is in $S_2^{N'}$;
\item There is a $p_{c}$-open circuit in $Ann(2^N,2^{N+1})$ which has a $p_{2^{N+1}}$-open connection to infinity, $k$ disjoint $\widetilde p_{i_0}$-closed dual paths from $\partial B(|\sigma_k|)$ to the dual circuit in $E_{i_0}^{\rho(i_0)}$ so that their intersections with $E_{i_0}^{\rho(i_0)}$ are in $S_4^{i_0}$ and $S_6^{i_0}$ respectively, and a $p_c$-open connection from $0$ to the defect in $D_{i_0}$ so that its intersection with $E_{i_0}^{\rho(i_0)}$ is in $S_2^{i_0}$ (only $S_4$'s are used if $k=1$).
\end{enumerate}
 See Figure \ref{Ei}.

We make two major claims about the events $\{K_\rho\}$:
\begin{enumerate}
\item[a.] There is a constant $c > 0$, uniform in $\rho$ and independent of $\kappa$ and $N$, such that
\[
\mathbb{P}(K_\rho) \geq c^{N' -i_0 + 1} \mathbb{P}_{cr}(A_{\sigma_k}(n)); 
\]
\item[b.] For $\rho \neq \rho'$, the events $K_\rho$ and $K_{\rho'}$ are disjoint.
\end{enumerate}
We first prove the $k \leq 2$ case of the theorem assuming the veracity of these claims. First we show that, on the event $K_\rho$, there are $k$ closed (noninvaded) dual arms from $\partial B(|\sigma_k|)$ to $\partial B(2^N)$. Indeed, assume the contrary. Then, a closed (at the appropriate parameter value) edge $e$ of one of the dual circuits from item 1 or of one of the closed dual arms from items 3 or 5 must be invaded. 

We show inductively in the invasion process that no such edge is invaded. Note that the origin is in the $\widetilde p_{i_0}$ infinite cluster, so the invasion will never take a $\widetilde p_{i_0}$-closed edge. Therefore no closed edge from the closed dual circuits or the closed dual arms is invaded before crossing the defect in $D_{i_0}$ (or ever). Assume now that no closed edge $e$ as above has been invaded by the time directly after we invade the defect in $D_i$, for some $i = i_0, \ldots, N'-1$. The invasion now has a $p_c$-open path to the defect in $D_{i+1}$, so does not invade such $e$ before reaching the ``near'' side of this defect. On the other hand, once it reaches this defect, it will invade the defect before any $\widetilde p_{i+1}$-closed edges. In particular, it will not invade any edge $e$ in one of the closed dual circuits in or dual arms leading to $D_j$, $i_0 \leq j \leq i+1$ before this defect, since all applicable $e$ are $\widetilde p_{i+1}$-closed.

Last, once the defect of the circuit in $D_{N'}$ is invaded, the process will only invade edges of the $p_{2^{N+1}}$-open infinite cluster. Since no $e$ as above is $p_{2^{N+1}}$-open, we see that each $K_\rho$ guarantees $A_{\sigma_k}(n)$. (In the case $k=2$, it is important that the closed arms joining the dual circuits in each $D_i$ intersect $S_i^4$ and $S_i^6$ respectively, so that they can be continued disjointly from circuit to circuit.) On the other hand, using the two claims, we have
\begin{align*}
  \mathbb{P}(A_{\sigma_k}(n)) \geq \mathbb{P}(\cup_\rho K_\rho) &= \sum_{\rho} \mathbb{P}(K_\rho)
  \\ &\geq \kappa^{N'-i_0+1} c^{N'-i_0+1} \mathbb{P}_{cr}(A_{\sigma_k}(n))\ ,
\end{align*}
where we have used that there are $\kappa^{N'-i_0+1}$ possible values of $\rho$. Choosing $\kappa = \lceil 2/c \rceil$ now proves the theorem.

It remains only to show that the two claims above hold. Claim b is easy, since if $\rho(i) \neq \rho'(i)$, then on $K_\rho$ there is a $p_c$-open path across $E_i^{\rho(i)}$, but no such crossing exists on $K_{\rho'}$. To see that Claim a holds, we can construct the event inductively using standard techniques. The outlet construction technique from Lemma~\ref{lma: outlet_construction} allows us to place a $\widetilde p_{i_0}$-closed circuit with the necessary defect in $E_{i_0}^{\rho(i_0)}$ and two additional $p_c$-open arms from it with probability at least some uniform $\delta$. In fact, by mimicking the construction of an outlet with $\tilde \sigma$-arms connecting to its closed dual circuit from the proof of \eqref{eq: tostada}, we can specifically guarantee in addition to the $\widetilde p_{i_0}$-closed dual circuit:
\begin{itemize}
\item  $k$ disjoint $\widetilde p_{i_0}$-closed dual connections from this circuit to $\partial B(2^{2\kappa i_0 + 2 \rho(i_0)})$ and $k$ disjoint $\widetilde p_{i_0+1}$-closed dual connections from this circuit to $\partial B(2^{2\kappa i_0 + 2 \rho(i_0) + 1})$ (remaining in $S_{i_0}^4$ and $S_{i_0}^6$ respectively if $k=2$), and
 \item  $p_c$-open connections from the defect to $\partial B(2^{2\kappa i_0 + 2 \rho(i_0) + 1})$ and $\partial B(2^{2\kappa i_0 + 2 \rho(i_0)})$ (remaining in $S_{i_0}^2$)
\end{itemize}
and that these connections are well-separated and ``extensible'' in the usual sense (as in arm direction techniques in \cite{kestenscaling}). Using generalized FKG allows us to connect the inner open path above to $0$ and the $\widetilde p_{i_0}$-closed dual arms to $\partial B(|\sigma_k|)$ at a cost of at most $c_0\, \mathbb{P}_{cr}(A_{\sigma_k}(2^{2\kappa i_0 + 2\rho(i_0)}))$ for some uniform $c_0$.

The construction of subsequent circuits is the same; each circuit constructed will cost a uniformly bounded constant probability, and connecting the $i$th to the $i+1$st  with the appropriate arms costs a uniform constant multiple of
\[
\mathbb{P}_{cr}\left(A_{\sigma_k}\left(2^{2 \kappa i + 2\rho(i) + 1}, 2^{2 \kappa (i+1) + 2\rho(i+1)} \right)\right)\ . 
\]
Note that this requires, as in our definition above, that the closed arms emanating from $E_i^{\rho(i)}$ are closed at parameter $\widetilde p_{i+1}$; otherwise, the constant would not be uniform over choices of $\rho$.

Therefore, we have for some uniform $c_1 > 0$
\begin{align*}
  \mathbb{P}\left(K_\rho\right) &\geq c_1^{N'-i_0+1} \mathbb{P}_{cr}\left(A_{\sigma_k}(2^{2\kappa i_0 + 2\rho(i_0)}) \right) \mathbb{P}_{cr}\left(A_{\sigma_k}(2^{2\kappa N' + 2\rho(N')+1}, 2^{N+1}) \right) \\
&\quad \times\prod_{i=0}^{N'} \mathbb{P}_{cr}\left(A_{\sigma_k}\left(2^{2 \kappa i + 2\rho(i) + 1}, 2^{2 \kappa (i+1) + 2\rho(i+1)} \right)\right)\\
&\geq c^{N'-i_0+1}\mathbb{P}_{cr}(A_{\sigma_k}(n))\ ,
\end{align*}
where in the last line we have used quasi-multiplicativity. This proves \eqref{eq: to_prove_outlet_jam}.

We now turn to the cases $k >2$. The strategy is a straight-forward combination of the proof of the case $k \leq 2$ and the proof of the first inequality of Theorem~\ref{thm: at_least_two}. Because the details are tedious and not so illuminating, we give here a sketch of the main idea.

\medskip
\noindent
{\bf Step 1.} Defining the events. The strategy is to define an event similar to $K_\rho$, from the previous case, but now starting at some scale $K$. That is, given large positive $K$ so that $B(|\sigma_k|) \subset B(K)$ and $\kappa > 0$, we set $i_0$ so that $B(4K) \subset B(2^{2\kappa i_0})$ and then put, for $n=2^N$, $N' = \lfloor N/2\kappa \rfloor -1$, where $N$ is large enough so that $N' > i_0$. We define $D_i$ and $E_i^a$ as before, and let $\rho = (\rho(i_0), \ldots, \rho(N')) \in \{0, 1, \ldots, \kappa-1\}^{N'-i_0+1}$. Now let $K_\rho = K_\rho(K,\kappa,N)$ be defined similarly as before, the only differences being that (a) we always take $k=2$ in the definition; that is, we always use two closed dual paths in all the items, (b) the $p_{2^N}$-closed dual connections in item 3 from the dual circuit in $E_{N'}^{\rho(N')}$ connect to the $p_c$ open circuit from item 5 in $Ann(2^N,2^{N+1})$, (c) in item 5, the $\widetilde p_{i_0}$-closed dual paths start at $\partial B(K)$ instead of $\partial B(|\sigma_k|)$ and the $p_c$-open path starts from $\partial B(K)$ instead of 0, and (d) in item 5, we place a $p_{2^N}$-closed dual circuit in $Ann(2^N,2^{N+1})$ with one defect (an edge which is not $p_{2^N}$-closed) in the interior of the $p_c$-open circuit.

Exactly the same arguments as in the last case imply that for some $c>0$, uniform in $\rho, \kappa, N$, and $K$, one has
\begin{equation}\label{eq: last_burrito_supreme}
\mathbb{P}(K_\rho) \geq c^{N'-i_0+1}\mathbb{P}_{cr}(A_{\sigma_2}(K,n)),
\end{equation}
and for $\rho \neq \rho'$, the events $K_\rho$ and $K_{\rho'}$ are disjoint.

\medskip
\noindent
{\bf Step 2.} Separating the closed paths. This is the most involved step of the proof. We must switch the strategy back to the one for the first inequality of Theorem~\ref{thm: at_least_two}. In other words, we need to make sure that for our given value of $k$, the closed arms in the definition of $K_\rho$ can be chosen in such a way that they are $k$-separated (they do not come within Euclidean distance $k$ of each other in $Ann(2K,2^{N-1})$). We begin by defining events similar to the $K_\rho$'s but with certain paths being $k$-separated.

Our separated event will be $K_\rho^*$, and it is a subevent of $K_\rho$. On $K_\rho$, make a measurable choice of the closed dual circuits in the $D_i$'s, and all the arms connecting either the circuits together, the innermost circuit to $\partial B(K)$, or the outermost circuit to the $p_c$-open circuit in $Ann(2^N,2^{N+1})$. We define these arms for $i=i_0, \ldots, N'$ as $\gamma_1^i,\gamma_2^i$, and $P^i$, where the $\gamma_j^i$'s are the closed arms connecting the circuit in $D_{i-1}$ to that in $D_i$ (in the case $i=i_0$, the arms start on $\partial B(K)$ and in the case $i=N'$, they end on the $p_c$-open circuit) and the $P^i$'s are the corresponding $p_c$-open arms connecting outlets. Write $x_i,y_i,z_i$ for the starting points of $\gamma_1^i, P^i$, and $\gamma_2^i$, and $x_i',y_i',z_i'$ for the ending points.

$K_\rho^*$ is defined as the subevent of $K_\rho$ on which the following occur:
\begin{enumerate}
\item possibly new closed arms $\hat \gamma_1^i$, $\hat \gamma_2^i$ can be selected such that they are $k$-separated between the circuits in $D_{i-1}$ and $D_i$ (in the case $i=i_0$, we only require them to be $k$-separated outside $B(2K)$ and in the case $i=N'$, inside $B(2^{N-1})$), and
\item defining $a_i,b_i$ for the starting points of $\hat \gamma_1^i$ and $\hat \gamma_2^i$ and $a_i',b_i'$ for the ending points, the arc of the circuit in $D_{i-1}$ (or $\partial B(K)$ in the case $i=i_0$) between $x_i$ and $z_i$ containing $y_i$ contains the corresponding arc between $a_i$ and $b_i$, and similarly for $x_i',y_i',z_i'$ and $a_i',b_i'$.
\end{enumerate}

The main argument of this step is to show that if $K$ is large enough, then for all $\kappa$ fixed large enough, and all $\rho, n$, one has
\begin{equation}\label{eq: new_separation}
\mathbb{P}(K_\rho) \leq 2 \mathbb{P}(K_\rho^*).
\end{equation}
The proof of this is similar to that of Claim~\ref{clam: clam_head}. We sketch the idea here.
\begin{proof}[Sketch of proof of \eqref{eq: new_separation}]
The proof is in two parts: first we need a result similar to Lemma~\ref{lma: lma_1}, stating that if we cannot choose such arms, then we can find an edge $f$ with a certain six-arm event. The second is to show that if $K$ is large enough, then the probability that this six arm event occurs in $Ann(2K,2^{N-1})$ conditional that $K_\rho$ occurs is at most 1/2.

For the first part, our version of Lemma~\ref{lma: lma_1} is the following. The only main difference is that the open arms can have defects, due to the presence of outlets in $K_\rho$. The proof is the same as in Lemma~\ref{lma: lma_1}: one chooses for the arms $\hat\gamma_1^i$, $\hat\gamma_2^i$ the extremal ones relative to $P_i$ (so that the region between them not containing $P_i$ is maximal). Such arms clearly satisfy item 2 above, and if they are not $k$-separated, one finds a six-arm event.
\begin{lma}\label{lma: lma_2}
For $n$ and $K$ as above, if $K_\rho$ occurs but $K_\rho^*$ does not, then there is an edge $f \in Ann(2K,2^{N-1})$ with the following six-arm conditions to distance $|f|/2$:
\begin{enumerate}
\item four disjoint $\widetilde p_i$-closed dual arms from $B(f,2k)$ (the rectangular box of sidelength $2k$ centered on the midpoint of $f$) to $\partial B(f,|f|/2)$, and 
\item two additional $\widetilde p_i$-open arms from $B(f,2k)$ to $\partial B(f,|f|/2)$, each with at most one defect. 
\end{enumerate}
\end{lma} 

Set $E_f$ to be the event described in Lemma~\ref{lma: lma_2}. Then one has
\[
\mathbb{P}(K_\rho \setminus K_\rho^*) \leq \sum_{f \in Ann(2K,2^{N-1})} \mathbb{P}(K_\rho, E_f).
\]
The second part of this step is to bound the sum by $(1/2) \mathbb{P}(K_\rho)$ if $K$ is large. This is accomplished similarly to before. One can show using independence and a gluing argument that there is $C$ such that
\[
\mathbb{P}(K_\rho,E_f) \leq C \mathbb{P}(K_\rho)\mathbb{P}(E_f).
\]
The argument is notationally much more complicated than before, but the idea is the same. One uses independence to upper bound by a product of three probabilities: $\mathbb{P}(E_f)$, the probability of an event comprised of the conditions of $K_\rho$ from $B(K)$ to $B(|f|/2)$, and an event comprised of the conditions of $K_\rho$ in $B(2|f|)^c$. Then one glues all the paths from the last two events, possibly reconstructing an outlet in $Ann(|f|/2,2|f|)$.

Still one has $\mathbb{P}(E_f) \leq |f|^{-2-\delta}$. The difference here is that the $E_f$ event now contains defected paths, but these only add factors of $\log^2 |f|$ (as in \cite[Proposition 18]{nolin}), and one can simply take $\delta$ smaller to compensate. Therefore one has
\begin{align*}
\mathbb{P}(K_\rho \setminus K_\rho^*) &\leq C \mathbb{P}(K_\rho) \sum_{f \in B(2K)^c} |f|^{-2-\delta} \\
&\leq (1/2) \mathbb{P}(K_\rho)
\end{align*}
if $K$ if fixed large enough.
\end{proof}

\medskip
\noindent
{\bf Step 3.} Forcing $K_\rho^*$ to imply $A_{\sigma_k}(n)$. For this last step, we want to ensure that the open paths connecting the outlets in $K_\rho^*$ are invaded, and that the regions between all the closed arms are not invaded. Since all of the closed arms are $k$-separated, we can then squeeze additional non-invaded arms between them, forcing $A_{\sigma_k}(n)$.

We do this just as in the arguments starting in the paragraph after \eqref{eq: last_banana}: we first note that $K_\rho^*$ implies the event $\hat K_\rho^*$, which is the same as the event $K_\rho^*$, but all the arms start from $\partial B(2K)$ and remain $k$-separated within $B(2^{N-1})$. Then we fix dual vertices $\hat x_1,\hat x_2$ and a vertex $\hat y_1$, all on $\partial B(2K)$ so that on $\hat K_\rho^*$, with positive probability, the open arm starts at $\hat y_1$ and the closed arms start at the $\hat x_i$'s. Next we define an event in $B(2K)$ similar to $\hat E$: define $E'$ as the event that the following conditions hold.
\begin{enumerate}
\item There is a $p_c$-open path from 0 to $\hat y_i$.
\item There are two disjoint $\widetilde p_{i_0}$-closed dual paths from dual neighbors of $0$ to the $\hat x_i$'s.
\item There is a $\widetilde p_{i_0}$-closed dual circuit around 0 in $Ann(3,|\sigma_k|)$ with one defect.
\item There are $k+1$ disjoint paths from $\partial B(|\sigma_k|)$ to $\partial B(2K)$. One of these is a portion of the $p_c$-open path from item 1, two of these are portions of the $\widetilde p_{i_0}$-closed dual paths from item 2, and $k-2$ of these have no weight restriction, and lie between the closed dual paths from item 2 in such a way that if we were to count these additional dual paths as closed, then there would be a $\sigma_k$-connection from $\partial B(|\sigma_k|)$ to $\partial B(2K)$.
\end{enumerate}
Since $K$ was fixed in the last step, we can find $C>0$ such that
\[
\mathbb{P}(K_\rho^* \cap E') = \mathbb{P}(K_\rho^*) \mathbb{P}(E') \geq C \mathbb{P}(K_\rho).
\]
On the other hand, $K_\rho^* \cap E'$ implies the invasion event $A_{\sigma_k}(2^{N-1})$. Indeed, on $K_\rho^* \cap E'$, we can follow two closed arms from neighbors of the origin up to $\partial B(2^{N-1})$. We follow arms from the origin to the circuit in $D_{i_0}$, take a portion of the circuit over to the closed arm from this circuit to the one in $D_{i_0+1}$, etc., all the way up to the closed arm from the circuit in $D_{N'}$ to the $p_{2^N}$-closed dual circuit in $Ann(2^N,2^{N+1})$. By the last step and the definition of $E'$, the arms are $k$-separated from $\partial B(|\sigma_k|)$ to the first circuit, from the second to the third, etc., up to $\partial B(2^{N-1})$. Due to item 2 in the definition of $K_\rho^*$ and the placement of the arms in the boxes $S_2^i, S_4^i, S_6^i$ from the definition of $K_\rho$, the portions of the circuits we take to build these arms are also $k$-separated. Therefore these two closed arms followed from $\partial B(|\sigma_k|)$ to $\partial B(2^{N-1})$ are $k$-separated.

Last, one can inductively prove (similar to the induction from the case $k\leq 2$) that the region between these two closed arms which does not contain any of the $p_c$-open arms is not invaded. Therefore one has 
\[
K_\rho^* \cap E' \subset A_{\sigma_k}(2^{N-1}).
\]
Noting that for distinct $\rho,\rho'$, the events $K_\rho^*\cap E'$ and $K_{\rho'}^*\cap E'$ are disjoint, we then finish with the same computation as in the last case, using \eqref{eq: last_burrito_supreme}:
\begin{align*}
\mathbb{P}(A_{\sigma_k}(2^{N-1})) \geq \mathbb{P}\left( \cup_\rho \left\{ K_\rho^* \cap E'\right\} \right) &= \sum_\rho \mathbb{P}(K_\rho^* \cap E') \\
&\geq C \sum_\rho \mathbb{P}(K_\rho) \\
&\geq C c^{N'-i_0+1} \kappa^{N'-i_0+1} \mathbb{P}_{cr}(A_{\sigma_2}(K,n)) \\
&\geq C c^{N'-i_0+1} \kappa^{N'-i_0+1} \mathbb{P}_{cr}(A_{\sigma_2}(n)).
\end{align*}
Taking $\kappa = \lceil 2/c \rceil$ finishes the proof.

\end{proof}

\subsection{Upper bound for $k=1,2$}
\label{sec: thm2-upper-bd}
Our aim is to prove the other bound of the second theorem: for $k=1,2$ and for some $\epsilon>0$, $C>0$,
\begin{equation}\label{eq: to_prove_kesten_zhang}
\mathbb{P}(A_{\sigma_k}(n)) \leq Cn^{-\epsilon} \mathbb{P}_{cr}(A_{\hat \sigma_k}(n)),
\end{equation}
where $\sigma_k$ is the sequence of one `$O$' and $k$ `$C$' entries, and $\hat \sigma_k$ is the sequence of zero `$O$' and $k$ `$C$' entries.

\begin{proof}


We begin with a lemma, which characterizes the event that there is a non-invaded path crossing an annulus in the invasion. (The complement is the event that there is an invaded circuit around 0.)
\begin{lma}\label{lma: inv-circuit}
For $1 \leq m \leq n$, there exists an open circuit around 0 in $Ann(m,n)$ in the invasion cluster $S$ if and only if there is $p \in [0,1]$ such that both of the following hold:
\begin{enumerate}
\item There is a $p$-open circuit $C$ around 0 in $Ann(m,n)$.
\item There is a $p$-closed dual circuit $C'$ around $C$.
\end{enumerate}
\end{lma}
\begin{proof}
First suppose that items 1 and 2 occur. Since the invasion must take an edge dual to $C'$, let $s$ be the first number such that $e_s$ (the $s$-th invaded edge) is in $C'$. At time $s-1$, the invasion must have intersected $C$ and therefore invaded all of $C$.

Conversely, suppose $C$ is an invaded circuit around $0$ in $Ann(m,n)$. Let $p' = \max\{t_e : e \in C\}$ and let $\hat e$ be the edge with $t_{\hat e} = p'$. At the moment directly before the invasion takes $\hat e$, the dual external edge boundary of the invasion contains a dual circuit $C'$ around 0 that is $p'$-closed. The dual external edge boundary of a set $V$ of vertices is the set {of} dual edges $e^*$ such that $e = \{v,w\}$, $v \in V$, $w \notin V$, and there exists an infinite vertex self-avoiding path in $\mathbb{Z}^2$ from $w$ that does not touch $V$. We make the following two observations:
\begin{enumerate}
\item \emph{$C'$ does not intersect $C$.} Since $C'$ is $p'$-closed and $C$ is $p'$-open, these circuits can only intersect at $\hat e$. If $C'$ contained $(\hat e)^*$, then $C'$ would cross $C$ at the edge $\hat e$ into the interior of $C$, but would not be able to cross again to reach the exterior. Thus they cannot intersect at all.
\item \emph{$C'$ is in the exterior of $C$.} At the moment after invading $\hat e$, the dual external edge boundary of the invasion still contains the dual circuit $C'$. This is because $\hat{e}$ is part of the circuit $C$, which lies entirely in the invasion. Hence both endpoints of $\hat{e}$ are already in the invasion graph at the step preceding the addition of $\hat{e}$. It follows that the dual external edge boundary does not change when $\hat{e}$ is added. By the preceding claim, either $C'$ is in the exterior of $C$, or $C$ it is separated from the origin by $C'$. The second case can be ruled out, since after taking $\hat{e}$ the invasion graph contains a path from the origin to $C$ which does not cross $C'$ by definition.
\end{enumerate}
Now we choose any $p$ satisfying
\[
p' < p < \min\{t_e : e \in C'\},
\]
so that $C$ is $p$-open but $C'$ is $p$-closed.
\end{proof}

Given the preceding Lemma characterizing the one dual closed arm event in the invasion, our strategy to derive \eqref{eq: to_prove_kesten_zhang} is as follows. First, we claim that the event $A_{\sigma_k}(n)$ in the invasion implies the event $A_{\hat{\sigma}_k}(|\sigma_k|,n,p_c)$. (Here $\hat{\sigma}_k$ is the sequence of $k$ `$C$' entries and zero `$O$' entries, so the parameter $p_c$ is only for the closed arms.)  To see this, if $k=1$, then $A_{\hat{\sigma}_k}(|\sigma_k|,n,p_c)^c$ means (by duality) that there is a $p_c$-open circuit around $0$ in $Ann(|\sigma_k|,n)$, and any such circuit must be invaded, so it is in $S$, the invasion graph, implying the event $A_{\sigma_k}(n)^c$. In the case $k=2$, $A_{\hat{\sigma}_k}(|\sigma_k|,n,p_c)^c$ means that there is a $p_c$-open circuit around 0 with at most one defect. But this circuit (minus its defect) will also have to be in the invasion, implying $A_{\sigma_k}(n)^c$.

Given the above claim, we write
\begin{align}
\mathbb{P}(A_{\sigma_k}(n)) &= \mathbb{P}(A_{\sigma_k}(n) \mid A_{\hat{\sigma}_k}(|\sigma_k|,n,p_c)) \mathbb{P}(A_{\hat{\sigma}_k}(|\sigma_k|,n,p_c)) \nonumber \\
&\leq C\mathbb{P}(A_{\sigma_k}(n) \mid A_{\hat{\sigma}_k}(|\sigma_k|,n,p_c)) \mathbb{P}_{cr} (A_{\hat{\sigma}_k}(n)) \label{eq: hamburger_cheeseburger},
\end{align}
and aim to show that the conditional probability is small. To do this, we will show that in each annulus $Ann(j) = Ann(3^j,3^{j+1})$, there is a uniformly positive conditional probability (given {{$\mathbb{P}$}} $A_{\hat{\sigma}_k}( |\sigma_k|,n,p_c)$) that there is $p> p_c$ such that there is a $p$-open circuit around zero enclosed by a $p$-closed circuit. We then must show that these events for different $j$ are almost independent, so that with probability exponentially high in the order $\log n$ number of annuli, at least one of them occurs, which forces $A_{\sigma_k}(n)$ not to occur. The independence claims will take some work to prove, and we will follow the strategy from Kesten-Zhang \cite[Theorem~2]{kestenzhang}. We will not look in every annulus $Ann(j)$, but in a constant fraction of them, and we will have to insulate different annuli from each other using $p_c$-closed circuits.

Proceeding with the details, for $K \geq 1$ (which will later be of order $\log_3 n$), define the random set
\[
J = \{j \in [1, K) : \exists ~p_c\text{-closed dual circuit around 0 both in } Ann(3j) \text{ and } Ann(3j+2)\}.
\]
Let $N = \#J$. By the FKG inequality, for any $C>0$,
\[
\mathbb{P}(A_{\hat{\sigma}_k}(|\sigma_k|,3^{3K}, p_c), N \leq CK) \leq \mathbb{P}(A_{\hat{\sigma}_k}(|\sigma_k|,3^{3K},p_c)) \mathbb{P}(N \leq CK).
\]
Since occurrences of $j \in J$ are independent for distinct $j$ RSW and standard concentration estimates show there are $C,a>0$ such that
\[
\mathbb{P}(N \leq CK) \leq e^{-aK}.
\]
Thus
\[
\mathbb{P}(A_{\hat{\sigma}_k}(|\sigma_k|,3^{3K},p_c), N \leq CK) \leq e^{-aK} \mathbb{P}(A_{\hat{\sigma}_k}(|\sigma_k|,3^{3K},p_c)).
\]
Using this, compute
\begin{align}
&\mathbb{P}(A_{\sigma_k}(3^{3K})) \nonumber \\
=~& \mathbb{P}(A_{\sigma_k}(3^{3K}), A_{\hat{\sigma}_k}(|\sigma_k|,3^{3K},p_c)) \nonumber \\
\leq~& \mathbb{P}(A_{\sigma_k}(3^{3K}), A_{\hat{\sigma}_k}(|\sigma_k|,3^{3K},p_c), N > CK) +e^{-aK} \mathbb{P}(A_{\hat{\sigma}_k}(|\sigma_k|,3^{3K},p_c)) \label{eq: taco_head}.
\end{align}
Looking back at \eqref{eq: hamburger_cheeseburger}, our goal is now to show that for some possibly smaller $a>0$,
\begin{equation}\label{eq: our_goal_nachos}
\mathbb{P}(A_{\sigma_k}(3^{3K}), A_{\hat{\sigma}_k}(|\sigma_k|,3^{3K},p_c), N > CK) \leq e^{-aK} \mathbb{P}(A_{\hat{\sigma}_k}(|\sigma_k|,3^{3K},p_c)).
\end{equation}
The purpose of including this event $\{N> CK\}$ is that we will need to use at least $CK$ dual circuits to decouple order $K$ annuli from each other. In these annuli we will build $p$-open and $p$-closed circuits, for different values of $p$.

For $j \in [1, K)$, let $E_j$ be the event that there is a $p_{3^{3j}}$-open circuit $C$ around 0 in $Ann(3j+1)$ and a $p_{3^{3j}}$-closed dual circuit around $C$ in $Ann(3j+1)$. Then by the previous lemma, one has 
\[
A_{\sigma_k}(3^{3K}) \subset \cap_{j \in J} E_j^c,
\]
where, as $J$ is random, $\cap_{j \in J}E_j^c$ is short-hand for $\bigcup_{\hat J} \left( (\cap_{j \in \hat J} E_j^c) \cap \{J = \hat J\}\right)$. So to prove \eqref{eq: our_goal_nachos}, we must show 
\begin{equation}\label{eq: taco_to_show}
\mathbb{P}\left( \cap_{j \in J} E_j^c, N > CK \mid A_{\hat{\sigma}_k}(|\sigma_k|,3^{3K},p_c) \right) \leq e^{-aK}.
\end{equation}

On the event $j \in J$, we can define $\mathcal{C}_{3j}$ and $\mathcal{D}_{3j+2}$ as the innermost and outermost $p_c$-closed dual circuits around 0 in the annuli $Ann(3j)$ and $Ann(3j+2)$ respectively. We obtain the decomposition
\begin{align*}
&\mathbb{P}(\cap_{j \in J} E_j^c, N > CK, A_{\hat{\sigma}_k}(|\sigma_k|,3^{3K},p_c)) \\
=~&\sum_{\hat J : \# \hat J > CK} \sum_{(C_{3j},D_{3j+2} : j \in \hat J)} \mathbb{P}\left(
\begin{array}{c}
\cap_{j \in J} E_j^c, ~J=\hat J, A_{\hat{\sigma}_k}(|\sigma_k|,3^{3K},p_c), ~\mathcal{C}_{3j} = C_{3j} \\
\text{ and } \mathcal{D}_{3j+2} = D_{3j+2} \text{ for } j \in \hat J.
\end{array} \right)
\end{align*}
Using independence, the inner probability is decomposed as the following (large) product, where we have written
\[
J = \{j(1) < \cdots < j(N)\}.
\]
Namely,
\begin{align}
&\mathbb{P}( \partial B(|\sigma_k|) \to_{\hat{\sigma}_k} C_{3j(1)}, ~1, \ldots, j(1)-1 \notin J, ~\mathcal{C}_{3j(1)} = C_{3j(1)}) \nonumber \\
\times~& \prod_{\ell =1}^N \mathbb{P}(C_{j(\ell)} \to_{\hat{\sigma}_k} D_{3j(\ell)+2},~E_{j(\ell)}^c) \label{eq: deep_inside}\\
\times~& \prod_{\ell=1}^{N-1} \mathbb{P}\left( 
\begin{array}{c}
D_{3j(\ell)+2} \to_{\hat{\sigma}_k} C_{3j(\ell+1)},~ \mathcal{D}_{3j(\ell)+2} = D_{3j(\ell)+2}\\
\mathcal{C}_{3j(\ell+1)} = C_{3j(\ell+1)},~ j(\ell)+1, \ldots, j(\ell+1)-1 \notin J
\end{array}\right) \nonumber \\
\times~& \mathbb{P}(D_{3j(N)+2} \to_{\hat{\sigma}_k} \partial B(3^{3K}),~ \mathcal{D}_{3j(N)+2} = D_{3j(N)+2},~ j(N)+1, \ldots \notin J)\nonumber .
\end{align}

The core of the argument, then, is to address the factors in \eqref{eq: deep_inside} and show:
\begin{prop}\label{prop: another_deep_inside}
There exists $\lambda < 1$ such that for all $j$, all self-avoiding dual circuits $C$ around 0 in $Ann(3j)$, and all self-avoiding dual circuits $D$ around 0 in $Ann(3j+2)$,
\[
\mathbb{P}(E_j^c \mid C \to_{\hat{\sigma}_k} D) \leq \lambda.
\]
\end{prop}

\begin{proof}
  We will argue by finding some sequence of events $(\Xi_j)$ with $\Xi_j \subseteq E_j$ and some $\epsilon$ uniform in $j, C, D$ such that $\mathbb{P} \left(\Xi_j \mid C \to_{\hat{\sigma}_k} D \right) \geq \epsilon$. Define the subboxes $B_j, \, B_j'$ of $Ann(3j + 1)$ by
\begin{align*}
B_j &= 3^{3j+1}(1 + 3^{-3}) e_1 + B(3^{3j-3})\\
B_j' &= 3^{3j+1}(1 + 3^{-3}) e_1 + B(3^{3j-4})\ .
\end{align*}
The closed arm(s) connecting $C$ to $D$ in the event $\{C \to_{\hat{\sigma}_k}D\}$ will be routed through $B_j$ to cross in $B_j'$ a four-arm edge, whose existence is furnished by the following claim.

\begin{claim}
  There exists $\delta > 0$ uniform in $j$ such that the following holds. Let $K_j$ be the event that there is a vertical edge $e = \{x, x + e_2\} \in B_j'$ with $t_e \in (p_c, p_{3^{3j}})$ having the following disjoint connections.
  \begin{itemize}
  \item A $p_c$-open arm from $x$ to the central portion of the bottom side of $B_j$ .
  \item A $p_{c}$-open arm from $x + e_2$ to the top side of $B_j$.
  \item A $p_c$-closed dual arm from each vertex of $e^*$, to the left and right sides of $B_j$ respectively.
  \end{itemize}
\end{claim}
\begin{proof}[Proof of claim]
Let $K_{j, e}$ denote the probability that the above connections occur from a particular edge $e \in B'_j$ (with the specified bounds on its edge weight), but now with the $p_c$-closed arms mandated to be $p_{3^{3j}}$-closed, and set
\[
N_j = \sum_{e \in B'_j} \mathbf{1}_{K_{j,e}}. 
\]
Standard arm-direction techniques \cite{kestenscaling} show as $j \to \infty$,
\[
\mathbb{P}(K_{j,e}) \asymp \mathbb{P}(A_{\tau}(3^{3j},p_c,p_{3^{3j}})) (p_{3^{3j}} - p_c)\ \text{uniformly in } e, 
\]
where $\tau$ is the sequence $(OCOC)$. Using \eqref{eq: changearms} and summing over $e$, we obtain 
\begin{align*}
  \mathbb{E}[N_j] \asymp (3^{3j})^2 \mathbb{P}_{cr}(A_{\tau}(3^{3j})) (p_{3^{3j}}-p_c).
\end{align*}
On the other hand, \cite[Prop. 34]{nolin} gives that
\[
(3^{3j})^2 \mathbb{P}_{cr}(A_{\tau}(3^{3j})) (p_{3^{3j}} - p_c) \asymp 1, 
\]
so $\mathbb{E}N_j \asymp 1$. 

Now note that for $e \neq e'$, the events $K_{j,e}$ and $K_{j,e'}$ are disjoint, so
\[
\mathbb{P}(\cup_e K_{j,e}) = \sum_e \mathbb{P}(K_{j,e}) = \mathbb{E}N_j \geq C
\]
for some $C>0$. Since $K_{j,e}$ implies the event in the claim, this completes the proof.

\end{proof}

We use the event $K_j$ to construct $\Xi_j$. Rather than give a precise construction, we just give the outline and a picture. Assume that the event $K_j$ occurs (along with all the conditioned connections outside of the annulus appearing in \eqref{eq: deep_inside}), and that a reflected version $K_j^*$ occurs, where this event is $K_j$ reflected about the $e_2$-axis, in corresponding reflected boxes $B_j^*$ and $(B_j')^*$. Then by independence and the previous lemma, $\mathbb{P}(K_j \cap K_j^*) > 0$ uniformly in $j$.

 By arm direction techniques \cite{kestenscaling} and gluing arguments, the closed dual arms from $e \in B'_j$ (and from an edge $f$ in the reflected $(B'_j)^*$ described in the event $K_j^*$) can be extended disjointly from their endpoints on $B_j$ and $B_j^*$ through the strip $\mathbb{R} \times [-3^{3j-3},3^{3j-3}]$ to connect to the circuits $C$ and $D$ listed in Proposition~\ref{prop: another_deep_inside} using the RSW theorem, with a net cost of at most a uniform constant in probability. Next, the $p_c$-open arms from $e$ and $f$ ending on the top and bottom boundaries of $B_j$ and $B_j^*$ can be completed into a $p_c$-open circuit with two defects ($e$ and $f$) around 0 in the thin annulus 
\[
Ann(3^{3j+1}(1+3^{-3}) - 3^{3j-3}, 3^{3j+1}(1+3^{-3}) + 3^{3j-3}).
\]
(This annulus is defined so that its inner and outer boundaries coincide with the left and right boundaries of $B_j$.) Using the generalized FKG inequality and the RSW theorem, this will also have at most a uniform constant cost. Note that the resulting $p_c$-open circuit with two defects is also $p_{3^{3j}}$-open. A larger thin annulus, still contained in $Ann(3j+1)$, will also be forced to contain a $p_{3^{3j}}$-closed circuit around 0. For specificity, we take this annulus to be
\[
Ann(3^{3j+1}(1+3^{-3}) + 3^{3j-3}, 3^{3j+1}(1+3^{-3}) + 2\cdot 3^{3j-3}).
\]
Using FKG and RSW again, this can be accomplished at cost of at  most another constant factor in probability. The final event illustrated in Figure~\ref{fig:fig_extend} is denoted $\Xi_j$. 

Whether $k=1$ or $2$, one has $\Xi_j \subset \{C \to_{\hat{\sigma}_k} D\} \cap E_j\ .$ The above line of reasoning also gives that $\Pr(\Xi_j) \geq \epsilon$ for $\epsilon$ sufficiently small, uniformly in $j$ and $C,\,D$. Taking $\lambda = 1 - \epsilon$ proves the claim.

\end{proof}

\begin{figure}
\centering
\includegraphics[scale = 0.35]{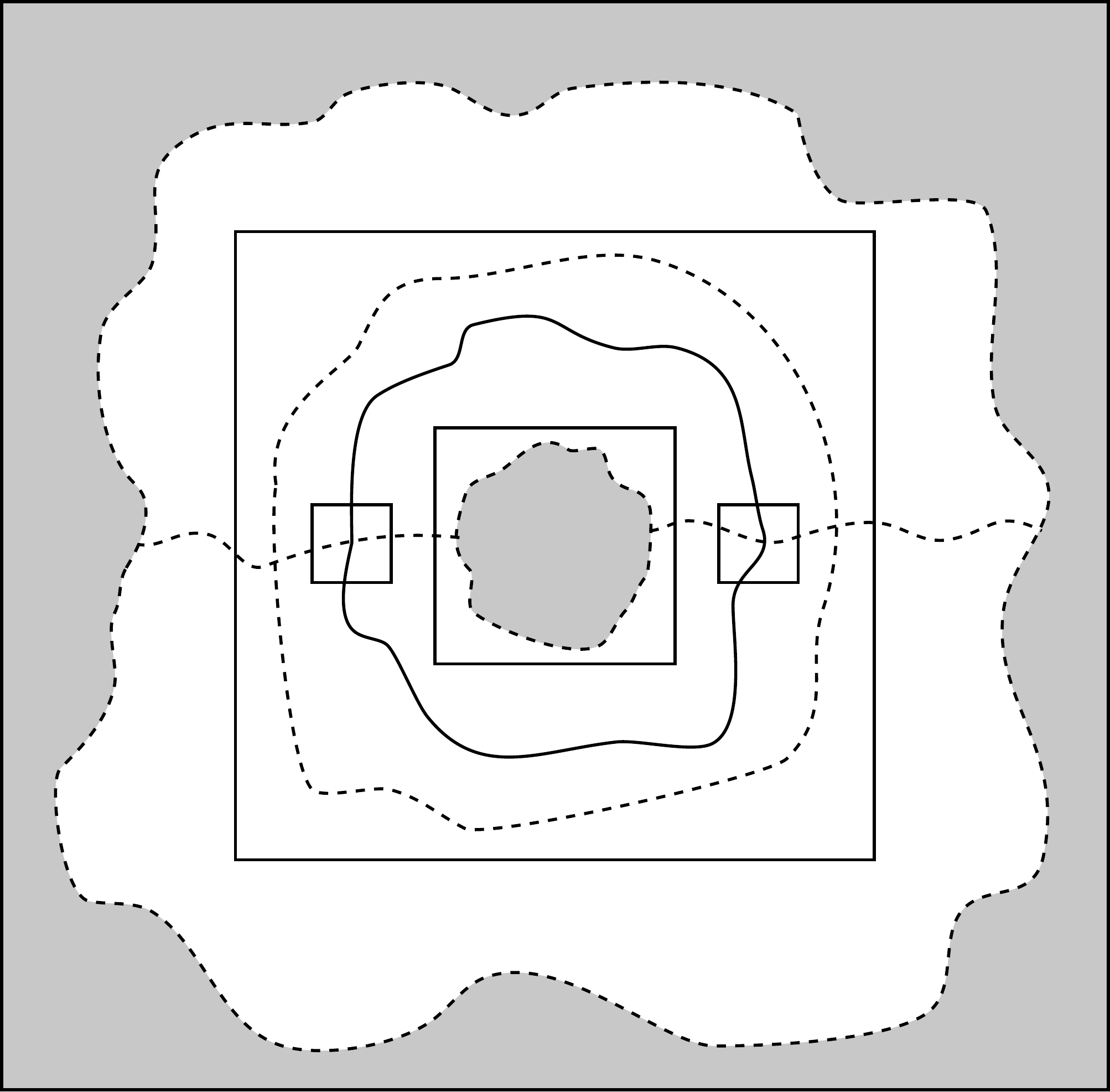}
\caption{Depiction of the extension to produce $\Xi_j$. The annulus $Ann(3j+1)$ is the region contained between the two smaller squares. The circuits $C$ and $D$ are the inner and outer dashed circuits; depicted as boundaries of shaded regions. These shaded regions represent parts of the configuration on which we have conditioned in our construction.}
\label{fig:fig_extend}
\end{figure}

At last, we return to \eqref{eq: deep_inside}, using the bound
\[
\mathbb{P}(C_{j(\ell)} \to_{\hat{\sigma}_k} D_{3j(\ell)+2}, E_{j(\ell)}^c) \leq \lambda \mathbb{P}(C_{j(\ell)} \to_{\hat{\sigma}_k} D_{3j(\ell)+2}),
\] 
and recombine everything (taking $a = -C \log \lambda$), to obtain the upper bound
\begin{align*}
&\mathbb{P}(\cap_{j \in J} E_j^c, N > CK, A_{\hat{\sigma}_k}(|\sigma_k|,3^{3K},p_c)) \\
\leq~&\sum_{\hat J : \# \hat J > CK} \sum_{(C_{3j},D_{3j+2} : j \in \hat J)} \lambda^{\# \hat J} \mathbb{P}\left(
\begin{array}{c}
J=\hat J, ~A_{\hat{\sigma}_k}(|\sigma_k|,3^{3K},p_c), ~\mathcal{C}_{3j} = C_{3j} \\
\text{ and } \mathcal{D}_{3j+2} = D_{3j+2} \text{ for } j \in \hat J.
\end{array} \right) \\
\leq~& \lambda^{CK} \mathbb{P}(A_{\hat{\sigma}_k}(|\sigma_k|,3^{3K},p_c), N > CK) \\
\leq~& e^{-aK} \mathbb{P}(A_{\hat{\sigma}_k}(|\sigma_k|,3^{3K},p_c)),
\end{align*}
which is \eqref{eq: taco_to_show}. This therefore proves \eqref{eq: our_goal_nachos}. Plugging into \eqref{eq: taco_head}, one obtains
\[
\mathbb{P}(A_{\sigma_k}(3^{3K})) \leq e^{-aK} \mathbb{P}(A_{\hat{\sigma}_k}(|\sigma_k|,3^{3K},p_c)) \leq C e^{-aK} \mathbb{P}(A_{\hat{\sigma}_k}(3^{3K})),
\]
and this is our version of a bound on \eqref{eq: hamburger_cheeseburger}.

To complete the proof of the main bound \eqref{eq: to_prove_kesten_zhang}, for a given $n$, take $K = \lfloor \frac{1}{3} \log_3 n \rfloor$, and use quasimultiplicativity:
\begin{align*}
\mathbb{P}(A_{\sigma_k}(n)) \leq \mathbb{P}(A_{\sigma_k}(3^{3K})) &\leq C e^{-aK} \mathbb{P}(A_{\hat{\sigma}_k}(3^{3K})) \\
&\leq C n^{-\epsilon} \mathbb{P}(A_{\hat{\sigma}_k}(n))
\end{align*}
for suitable $\epsilon>0$.
\end{proof}



\end{document}